\documentclass[12pt,a4paper]{article}

\usepackage[utf8]{inputenc}
\usepackage[T1]{fontenc}
\usepackage{lmodern}

\usepackage[english]{babel}

\usepackage[pdftex]{graphicx}

\usepackage{amsmath}
\usepackage{amssymb}
\usepackage{amsthm}
\usepackage{anysize}
\usepackage{accents}

\usepackage{hyperref}
\usepackage[normalem]{ulem}
\usepackage[stable]{footmisc}

\theoremstyle{plain}
\newtheorem{thm}{Theorem}[section]
\newtheorem*{thm*}{Theorem}
\newtheorem{lem}[thm]{Lemma}
\newtheorem{prop}[thm]{Proposition}
\newtheorem{clm}[thm]{Claim}
\newtheorem{cor}[thm]{Corollary}
\newtheorem{conj}[thm]{Conjecture}

\theoremstyle{definition}
\newtheorem{df}[thm]{Definition}
\newtheorem{denot1}[thm]{Notation}
\newtheorem{denot}[thm]{Notations}
\newtheorem{assm}[thm]{Assumption}
\newtheorem{obs}[thm]{Observation}

\theoremstyle{remark}
\newtheorem{rem}[thm]{Remark}
\newtheorem{exmp}[thm]{Example}

\numberwithin{equation}{section}

\newcommand{\emd}[1]{\emph{#1}}

\newcommand{\Isom}{\mathrm {Isom}}
\newcommand{\id}{\mathrm {Id}}
\newcommand{\ud}{\mathrm{d}}
\newcommand{\sm}{\setminus}
\newcommand{\G}{G} 
\renewcommand{\S}{{\mathbb{S}}}
\newcommand{\HH}{{\mathbb{H}^2}}
\newcommand{\HHH}{{\mathbb{H}^3}}
\newcommand{\Hd}{{\mathbb{H}^d}}
\renewcommand{\c}[1]{\widehat{#1}}
\newcommand{\bdi}{\partial\,}
\newcommand{\bdt}{\mathrm{bd}\,}
\newcommand{\bd}{\mathrm{bd}\,}
\newcommand{\sint}{\mathrm{int}\,}
\newcommand{\cl}[1]{{\overline{#1}}}
\newcommand{\clcX}[1]{{\overline{#1}^\cX}}
\newcommand{\cX}{{\hat{X}}}
\newcommand{\cHd}{{\widehat{\mathbb{H}}^d}}
\newcommand{\hc}[1]{\accentset{\frown}{#1}}
\newcommand{\hcHd}{{\hc{\mathbb{H}}^d}}
\newcommand{\hbd}{\eth}
\newcommand{\pbd}{\pi} 
\newcommand{\N}{{\mathbb{N}}}
\newcommand{\R}{{\mathbb{R}}}
\newcommand{\Z}{{\mathbb{Z}}}
\renewcommand{\L}{\mathcal{L}}
\newcommand{\NP}{{\mathcal{N}}}
\newcommand{\p}{{p_{1/2}}}
\newcommand{\hR}{^{(h,R)}}
\newcommand{\ho}{o\hR}
\newcommand{\IOd}{(0;1]\times O(d)}
\newcommand{\gt}{\tilde{g}}
\newcommand{\CG}{\Gamma}
\newcommand{\al}{\alpha}
\renewcommand{\b}{\beta}
\newcommand{\g}{\gamma}
\renewcommand{\d}{\delta}
\newcommand{\e}{\varepsilon}
\renewcommand{\phi}{\varphi}
\renewcommand{\rho}{\varrho}
\newcommand{\s}{\sigma}
\renewcommand{\Pr}{\mathrm {P}}
\newcommand{\Est}{\mathrm {E}}
\newcommand{\pc}{{p_\mathrm c}}
\newcommand{\pu}{{p_\mathrm u}}
\newcommand{\op}{\omega}
\newcommand{\conn}{\leftrightarrow}
\newcommand{\then}{\Longrightarrow}

\newcommand{\Bo}{B^1_{r_0}}
\newcommand{\bo}{\mathbf{o}}
\newcommand{\vh}{v_\mathrm{h}}
\newcommand{\tends}[1]{\xrightarrow[#1]{}}
\newcommand{\RhR}{\mathcal{R}\hR}

\renewcommand{\(}{\left(}
\renewcommand{\)}{\right)}

\newlength{\bigcupdotwidth}
\newcommand{\bigcupdotsymb}{\makebox[\bigcupdotwidth]
  {\makebox[0pt]{$\displaystyle{\bigcup}$}\makebox[0pt]{$\cdot$}}}
\DeclareMathOperator*{\bigcupdot}{\bigcupdotsymb}

\newlength{\cupdotwidth}
\newcommand{\cupdotsymb}{\makebox[\cupdotwidth]
  {\makebox[0pt]{$\displaystyle{\cup}$}\makebox[0pt]{$\cdot$}}}
\DeclareMathOperator*{\cupdot}{\cupdotsymb}

\newcommand{\voidindop}[2]{#1_{\makebox[0pt]{$\scriptstyle#2$}}}

\newlength{\rvoidindopOPWD}
\newlength{\rvoidindopAUX}

\newcommand{\rvoidindop}[2]{%
\settowidth{\rvoidindopOPWD}{$#1$}%
\settowidth{\rvoidindopAUX}{$\scriptstyle#2$}%
\addtolength{\rvoidindopAUX}{-\rvoidindopOPWD}%
\rule{0.5\rvoidindopAUX}{0mm}%
#1_{\makebox[0pt]{$\scriptstyle#2$}}}

\newlength{\bigcapKXwidth}
\newcommand{\bigcapKX}{\makebox[\bigcapKXwidth]{$\displaystyle{\bigcap_{\substack{K\subseteq X\\ K\text{ -- compact}}}}$}}

\author{Jan Czajkowski\\
\small Argentinian National Research Council at the University of Buenos Aires\\
}
\title{One-point boundaries of ends of clusters\linebreak{}in percolation in $\Hd$}

\begin{document}

\maketitle
\thispagestyle{empty}

\begin{abstract}
Consider Bernoulli bond percolation on a graph nicely embedded in hyperbolic space $\Hd$ in such a way that it admits a transitive action by isometries of $\Hd$. Let $p_0$ be the supremum of such percolation parameters that no point at infinity of $\Hd$ lies in the boundary of the cluster of a fixed vertex with positive probability. Then for any parameter $p < p_0$, a.s.\ every percolation cluster has only one-point boundaries of ends.
\end{abstract}
\section{Introduction}

Consider a graph $\CG$. Fix $p\in[0,1]$ and for each edge, mark it as ``open'' with probability $p$; do it independently for all edges of $\CG$. We mark the edges which are not declared open, as ``closed''. The \emd{state} of an edge is the information whether it is open or closed.
The (random) set $\op$ of open edges of $\CG$ forms, together with all the vertices of $\CG$, a random subgraph of $\CG$. We will often identify it with the set of edges $\op$. This is a model of percolation, which we call \emd{Bernoulli bond percolation} on $\CG$ \emd{with parameter} $p$.

One of the main objects of interest in percolation theory are
the connected  components of the random subgraph $\op$, called \emd{clusters}, and their ``size''. Here, ``size'' can mean the number of vertices, the diameter and many other properties of a cluster which measure how ``large'' it is. For example, one may ask if there is a cluster containing infinitely many vertices in the random subgraph, with positive probability.
\begin{df}

We define the \emd{critical probability} for the above percolation model on $\CG$ as
$$\pc(\CG):=\inf\{p\in[0,1]:\textrm{with positive probability }\op\textrm{ has some infinite cluster}\}.$$
\end{df}
It follows from the Kolmogorov 0-1 law that if the graph $\CG$ is connected and locally finite, then the probability
$\Pr(\op\textrm{ has some infinite cluster})$ always
equals $0$ or $1$. Since this event is an \emd{increasing} event (see Definition \ref{incr}), its probability is an increasing function of $p$.
Thus, it is $0$ for $p<\pc(\CG)$ and $1$ for $p>\pc(\CG)$.
\begin{df}
A graph $\CG$ is
\emd{transitive}, if its automorphism group acts transitively on the set of vertices of $\CG$. $\CG$ is \emd{quasi-transitive}, if there are finitely many orbits of vertices of $\CG$ under the action of its automorphism group.
\end{df}
If $\CG$ is connected, locally finite and quasi-transitive, then it is known from \cite[Theorem~1]{NewmSchul} that the number of infinite clusters in $\op$ is
a.s.~constant and equal to
$0$, $1$ or $\infty$ (see also \cite[thm.~7.5]{LP}). Let us focus on the question when this number is $\infty$.
\begin{df}
The \emd{unification probability} for the above percolation model on any graph $\CG$ is the number
$$\pu(\CG):=\inf\{p\in[0,1]:\textrm{a.s.~there is a unique infinite cluster in $\op$}\}.$$
\end{df}

\begin{df}
A locally finite graph $\CG$ is \emd{non-amenable} if there is a constant $\Phi>0$ such that for every non-empty finite set $K$ of vertices of $\CG$, $|\bd K|\ge\Phi|K|$, where $\bd K$ is the set of edges of $\CG$ with exactly one vertex in $K$ (a kind of boundary of $K$).
Otherwise, $\CG$ is \emd{amenable}.
\end{df}

It turns out that if a connected, transitive, locally finite graph $\CG$ is amenable, then there is at most $1$ infinite cluster a.s.\ for bond and site Bernoulli percolation, see \cite[thm.~7.6]{LP}.
So, if $\pc(\CG)<\pu(\CG)$ then the graph $\CG$ must be non-amenable.
It is an interesting question whether the converse is true:
\begin{conj}[\cite{BS96}]\label{conjBS96}
If $\CG$ is non-amenable and quasi-transitive, then $\pc(\CG)<\pu(\CG)$.
\end{conj}
There are several classes of non-amenable graphs, for which the above conjecture has been established. Let us mention here a few of them. For an infinite regular tree it is actually folklore.
It was shown for bond percolation on the Cartesian product of $\mathbb{Z}^d$ with an infinite regular tree of sufficiently high degree in \cite{GrimNewm}.
Later, it was shown for site percolation on Cayley graphs of a wide class of Fuchsian groups in \cite{Lal}, and
for site and bond percolation on transitive, non-amenable, planar graphs with one end in \cite{BS01}.
(A graph has \emd{one end}, if after throwing out any finite set of vertices it has exactly one infinite component.) These two result concern the hyperbolic plane $\HH$. Similarly, the conjecture is obtained in \cite{Cz3ph} for many tiling graphs in $\HHH$.
There is also a rather general result in \cite{PSN} saying that any finitely-generated non-amenable group has a Cayley graph $\CG$ with $\pc(\CG)<\pu(\CG)$ for bond percolation.

An interested reader may consult e.g.\ \cite{Grim} and \cite{LP}, which give quite wide introduction to percolation theory.

\subsection{Boundaries of ends}
In this paper we consider percolation clusters on graphs ``naturally'' embedded in $\Hd$ with $d\ge 2$.
We define the \emd{boundaries of ends} of a cluster in $\Hd$ as follows:
\begin{denot}
For any topological space $X$, by $\sint_X$ and $\overline{\makebox[1em]{$\cdot$}}^X$ we mean operations of taking interior and closure, respectively, in the space $X$.
\end{denot}
\begin{df}\label{dfbd}
Let $X$ be a completely regular Hausdorff ($\mathrm T_{3\frac{1}{2}}$), locally compact topological space. Then:
\begin{itemize}
\item An \emd{end} of a subset $C\subseteq X$ is a function $e$ from the family of all compact subsets of $X$ to the family of subsets of $C$ such that:
\begin{itemize}
 \item for any compact $K\subseteq X$ the set $e(K)$ is one of the component of $C\setminus K$;
 \item for $K\subseteq K'\subseteq X$ -- both compact -- we have
 $$e(K)\supseteq e(K').$$
\end{itemize}
\end{itemize}
Now let $\cX$ be an arbitrary compactification of $X$. Then
\begin{itemize}
\item The \emd{boundary} of $C\subseteq X$ is the following:
$$\bdi C= \clcX{C}\setminus X.$$
\item Finally the \emd{boundary of an end} $e$ of $C\subseteq X$ is
$$\bdi e=\bigcapKX\bdi e(K).$$
\end{itemize}
We also put $\c C = \clcX{C}$. Whenever we use the usual notion of boundary (taken in $\Hd$ by default), we denote it by $\bdt$ to distinguish it from $\bdi$.
\par We use these notions in the context of the hyperbolic space $\Hd$, where the underlying compactification is the compactification $\cHd$ of $\Hd$ by its set of points at infinity\footnote{For $\Hd$, it is the same as its Gromov boundary---see \cite[Section III.H.3]{BH}.} (see \cite[Definition II.8.1]{BH}). The role of $C$ above will be played by percolation clusters in $\Hd$.
\par Thus, $\bdi\Hd =\cHd\sm\Hd$ is the set of points at infinity. If $\Hd$ is considered in its Poincar\'e disc model\footnote{It is called also Poincar\'e ball model.}, $\bdi\Hd$ is naturally identified with the boundary sphere of the Poincar\'e disc.
\end{df}
\begin{rem}\label{smarthat}
In this paper, whenever we consider a subset of $\Hd$ denoted by a symbol of the form e.g.\ $C_x^y(z)$, we use the notation  $\c C_x^y(z)$ for its closure in $\cHd$ instead of $\c{C_x^y(z)}$, for aesthetic reasons.
\end{rem}
Let us define a percolation threshold $\p$ as the supremum of percolation parameters $p$ such that $\Pr_p$-a.s.\ all infinite clusters in $\omega$ have only one-point boundaries of ends.
The question is if $\pc<\p<\pu$ e.g.\ for some natural tiling graphs in $\Hd$ for $d\ge 3$. In such case one will have an additional percolation threshold in the non-uniqueness phase.
In this paper we give a sufficient condition for $p$-Bernoulli bond percolation\ to have only one-point boundaries of ends of infinite clusters, for a large class of transitive graphs embedded in $\Hd$. That sufficient condition is ``$p<p_0$'', where $p_0$ is a threshold defined in Definition \ref{dfp0}. The key part of the proof is an adaptation of the proof of Theorem (5.4) from \cite{Grim}, which in turn is based on \cite{Men}.
\par In the next section we formulate the assumptions on the graph and the main theorem.

\subsection{The graph and the sufficient condition}

\begin{df}
Let for any graph $\CG$, $V(\CG)$ denote its set of vertices and $E(\CG)$ its set of edges. In this paper, we often think of a $\op\subseteq E(\CG)$ as a sample, called \emd{percolation configuration}. Accordingly, $2^{E(\CG)}$ is the sample space for modelling Bernoulli bond percolation. The accompanying $\s$-algebra on it consists of all Borel sets (with respect to the product topology). Taking the natural product measure $\Pr_p$ on $2^{E(\CG)}$, we treat the configuration $\op$ as the random (set-valued) variable described above.

For any graph $\G$ embedded in arbitrary metric space, we call this embedded graph \emd{transitive under isometries} if some group of isometries of the space acts on $\G$ by graph automorphisms transitively on its set of vertices.

A graph embedding in a topological space is \emd{locally finite} if every compact subset of $\Hd$ meets only finitely many vertices and edges of the embedded graph.

By a \emd{simple} graph we mean a graph without loops and multiple edges.
\end{df}
\begin{assm}\label{assmG}
Throughout this paper{} we assume that $\G$ is a connected (simple) graph embedded in $\Hd$, such that:
\begin{itemize}
\item its edges are geodesic segments;
\item the embedding is locally finite;
\item it is transitive under isometries.
\end{itemize}
Let us also pick a vertex $o$ (for ``origin'') of $\G$ and fix it once and for all.
\end{assm}

Note that by these assumptions, $V(\G)$ is countable, $\G$ has finite degree and is a closed subset of $\Hd$.

\begin{df}\label{dfp0}
For $v\in V(\G)$, by $C(v)$ we mean the percolation cluster of $v$ in $\G$.
Let $\NP(\G)$ (for ``null''), or $\NP$ for short, be defined by
\begin{equation}\label{assmnull}
\NP(\G) = \{p\in[0;1]:(\forall x\in\bdi\Hd)(\Pr_p(x\in\bdi C(o))=0)\}
\end{equation}
and put
\begin{equation}
p_0=p_0(\G)=\sup\NP(\G).
\end{equation}
\end{df}
\begin{rem}
In words, $\NP$ is the set of parameters $p$ of Bernoulli bond percolation{} on $G$ such that~no point of $\bdi\Hd$ lies in boundary of the cluster of $o$ with positive probability.
Note that $\NP$ is an interval (the author does not know whether it is right-open or right-closed) because the events $\{x\in\bdi C(o)\}$ for $x\in\bdi\Hd$ are all increasing (see Definition \ref{incr}), so $\Pr_p(x\in\bdi C(o))$ is a non-decreasing function of $p$ (see \cite[Theorem (2.1)]{Grim}). That allows us to think of $p_0$ as the point of a phase transition.
\par We are going to make a few more remarks concerning the above definition and how $p_0$ may be related to the other percolation thresholds in Section \ref{secremsp0}.
\end{rem}

Now, we formulate the main theorem:
\begin{thm}\label{mainthm}
Let $G$ satisfy the Assumption \ref{assmG}. Then, for any $0\le p<p_0$, a.s. every cluster in $p$-Bernoulli bond percolation\ on $\G$ has only one-point boundaries of ends.
\end{thm}
The key ingredient of the proof of this theorem is Lemma \ref{corMen}, which is a corollary of Theorem \ref{lemMen}. The latter is quite interesting in its own right. They are presented (along with a proof of Lemma \ref{corMen}) in separate Section \ref{secexpdec}. The elaborate proof of Theorem \ref{lemMen}, rewritten from the proof of Theorem (5.4) in \cite{Grim}, is deferred to Section \ref{prflemMen}. The proof of this theorem itself, is presented in Section \ref{secprfmainthm}.
\begin{rem}
In the assumptions of the above theorem, $p_0$ can be replaced by
$$p_0' = \sup\{p\in[0;1]: g_p(r)\tends{r\to\infty}0\}$$
with $g_p(r)$ from the Definition \ref{gpr} because only the fact that for $p<p_0$, $g_p(r)\tends{r\to\infty}0$ (implied by Claim \ref{Mtight}) is used (in the proof of Lemma \ref{gt<=sqrt}). Accordingly, $p_0'\ge p_0$. Nevertheless, the author does not know if it is possible that $p_0'>p_0$.
\end{rem}

\subsection{Remarks on the sufficient condition}\label{secremsp0}

In this section we give some remarks on the threshold $p_0$ and on the events $\{x\in\bdi C(o)\}$ (used to define $\NP$).
\begin{df}
For $A,B\subseteq\Hd$, let $A\conn B$ be the event that there is an open path in the percolation process (given by the context) intersecting both $A$ and $B$. We say also that such path \emd{joins} $A$ and $B$. If any of the sets is of the form $\{x\}$, we write $x$ instead of $\{x\}$ in that formula and those phrases.
\end{df}
\begin{rem}\label{nullmeasb}
For $x\in\bdi\Hd$, the configuration property $\{x\in\bdi C(o)\}$ is indeed a (measurable) random event.
Even more: the set
\begin{equation}
A = \{(x,\omega)\in \bdi\Hd\times 2^{E(\G)}:x\in\bdi(C(o))(\omega)\}
\end{equation}
is measurable in the product $\bdi\Hd\times 2^{E(\G)}$ (where the underlying $\s$-field on $\bdi\Hd$ is the $\s$-field of Borel sets). To prove it, let us introduce a countable family $(H_n)_{n\in\N}$ of half-spaces such that~the family of open discs
$$\{\sint_{\bdi\Hd}\bdi H_n:n\in\N\}$$
is a base of the topology on $\bdi\Hd$. Then, let us rewrite the condition defining $A$:
\begin{eqnarray}
x\in \bdi C(o) &\iff& (\forall n)(x\in\sint_{\bdi\Hd}\bdi H_n \then C(o)\cap H_n\neq\emptyset) \iff\\
&\iff& (\forall n)\bigl(\neg(x\in\sint_{\bdi\Hd}\bdi H_n) \lor\\
&&\lor\, \bigl(x\in\sint_{\bdi\Hd}\bdi H_n \land (\exists v\in V(\G)\cap H_n)(o\conn v)\bigr)\bigr),\notag
\end{eqnarray}
which is a measurable condition, as the sets
$$\{(x,\omega)\in \bdi\Hd\times 2^{E(\G)}: x\in\sint_{\bdi\Hd}\bdi H_n\} = \sint_{\bdi\Hd}\bdi H_n\times 2^{E(\G)}$$
and
$$\{(x,\omega)\in \bdi\Hd\times 2^{E(\G)}: o\conn v\textrm{ in }\omega\}$$
are measurable for $n\in\N$, $v\in V(\G)$.
\par For $x\in\bdi\Hd$, the measurability of the event $\{x\in\bdi C(o)\}$ follows the same way if we treat it as the $x$-section of $A$:
\begin{equation}
\{x\in\bdi C(o)\} = \{\omega:(x,\omega)\in A\}.
\end{equation}

\end{rem}

\begin{rem}\label{ineqp0}
The threshold $p_0$ is bounded as follows:
$$\pc\le p_0 \le\pu.$$
The inequality $\pc\le p_0$ is obvious and the inequality $p_0\le\pu$ can be shown as follows: if $p$ is such that $\Pr_p$-a.s.~there is a unique infinite cluster in $\G$, then with some probability\ $a>0$, $o$ belongs to the infinite cluster and by BK-inequality (see Theorem \ref{BKineq}), for any $v\in V(\G)$,
$$\Pr_p(o\conn v)\ge a^2.$$

Take $x\in\bdi\G$. Choose a decreasing (in the sense of set inclusion) sequence $(H_n)_n$ of half-spaces such that $\bigcap_{n=1}^\infty \sint_{\bdi\Hd}\bdi H_n=\{x\}$. Since $x\in\bdi\G$, we have $V(\G)\cap H_n \neq \emptyset$ for all $n$. Therefore
\begin{align}
\Pr_p(x\in\bdi C(o)) &= \Pr_p\left(\bigcap_{n\in\N}\{(\exists v\in V(\G)\cap H_n)(o\conn v)\}\right) =\\
&= \lim_{n\to\infty} \Pr_p((\exists v\in V(\G)\cap H_n)(o\conn v)) \ge a^2.
\end{align}
Hence, $p\notin\NP$, so $p\ge p_0$, as desired.
\par The main theorem (Theorem \ref{mainthm}) is interesting when $\pc<p_0$. The author does not know what is the class of embedded graphs $\G$ (even among those arising from Coxeter reflection groups as in \cite{Cz3ph}) satisfying $\pc(\G)<p_0(\G)$. The author suspects that $p_0=\pu$ for such graphs as in \cite{Cz3ph} in the cocompact case (see Remark \ref{meas0->nullae}; in such case most often we would have $p_0>\pc$). On the other hand, there are examples where $p_0<\pu$ (see Example \ref{p0<pu} below). Still, the author does not know if it is possible that $\pc=p_0<\pu$.
\end{rem}
\begin{exmp}\label{p0<pu}
Let $\Pi$ be an unbounded polyhedron with $6$ faces in $\HHH$ whose five faces are cyclically perpendicular and the sixth one is disjoint from them. Then, the group $G$ generated by the (hyperbolic) reflections in the faces of $\Pi$ is isomorphic to the free product of $\mathbb Z_2$ and the group $G_5 < \Isom(\HH)$ generated by the reflections in the sides of a right-angled pentagon in $\HH$. Let $\CG$ and $\CG_5$ be the Cayley graphs of $G$ and $G_5$, respectively. Then, $\CG$ has infinitely many ends, so from \cite[Exercise 7.12(b)]{LP} $\pu(\CG)=1$. Next, if $p>\pu(\CG_5)$, then with positive probability\ $\bdi C(o)$ contains the whole circle $\bdi(G_5\cdot o)$. (It is implied by Theorem~4.1 and Lemma~4.3 from \cite{BS01}.) Hence, $p_0\le\pu(\CG_5)<\pu(\CG)$, as $\pu(\CG_5)<1$ by \cite[Theorem~10]{BB}. Moreover, the conclusion of the main theorem (Theorem \ref{mainthm}) fails for any $p>\pu(\CG_5)$.
\end{exmp}
\begin{rem}\label{meas0->nullae}
This remark is hoped to explain a little the suspicion that for the Cayley graph of a cocompact Coxeter reflection group in $\Hd$, we have $p_0=\pu$ (Remark \ref{ineqp0}). It is based on another suspicion: for $p<\pu$ in the same setting,
\begin{equation}\label{|bd|=0}
\Pr_p\textrm{-a.s.~}|\bdi C(o)|=0,
\end{equation}
which is a property of the $p$-Bernoulli bond percolation quite similar to $\Pr_p(x\in\bdi C(o))=0$.
Here $|\cdot|$ can be the Lebesgue measure on $\bdi\Hd=\S^{d-1}$, or the Poisson measure on $\bdi\Hd$ arising from the simple random walk on $\G$ starting at $o$. (For a definition of a simple random walk and an explanation of Poisson boundary, see \cite{Woe}, Section 1.C and Section 24., p.~260, respectively.) If one proves \eqref{|bd|=0}, then the probability vanishing in \eqref{assmnull} follows for $|\cdot|$-a.e. point $x\in\bdi\Hd$. In addition, because the induced action of such cocompact group on $\bdi\Hd$ has only dense orbits (see e.g.~\cite[Proposition 4.2]{KapBen}), one might suspect that in such situation as above, $\Pr_p(x\in\bdi C(o))=0$ holds for all $x\in\bdi\Hd$.
\end{rem}

\newcommand{\maled}{d}
\section{Definitions: percolation on a fragment of $\mathbb{H}^{\lowercase{d}}$}

Here we are going to introduce some notions and notations used in Theorem \ref{lemMen} and Lemma \ref{corMen} and in the proof of the main theorem.
\begin{denot1}
Let us adopt the convention saying that natural numbers include $0$. We denote the set of all positive natural numbers by $\N_+$.
\end{denot1}
\begin{df}\label{dfhHd}
For the rest of this paper, consider $\Hd$ in its fixed half-space Poincar\'e model (being the upper half-space $\R^{d-1}\times(0;\infty)$) in which the point $o$ (the distinguished vertex of $\G$) is represented by $(0,\ldots,0,1)$. (It will play the role of origin of both $\Hd$ and $\G$.)
\par The half-space model of $\Hd$ and its relation to the Poincar\'e ball model are explained in \cite[Chapter I.6, p.~90]{BH}. Note that the inversion of $\R^d$ mapping the Poincar\'e ball model $\mathbb{B}^d$ to our fixed half-space model sends one point of the sphere $\bdt\mathbb{B}^d$ to infinity. In the context of the half-space model, we treat that ``infinity'' as an abstract point (outside $\R^d$) compactifying $\R^d$. We call it the \emd{point at infinity} and denote it by $\infty$.
\par Let $\hcHd$ be the closure of $\Hd$ in $\R^d$ and $\hbd\Hd=\hcHd\sm\Hd$ 
(so here $\hcHd=\R^{d-1}\times[0;\infty)$ and $\hbd\Hd=\R^{d-1}\times\{0\}$). We identify $\hcHd$ with $\cHd\sm\{\infty\}$ and $\hbd\Hd$ with $\bdi\Hd\sm\{\infty\}$ in a natural way. Also, for any closed $A\subseteq\Hd$, let $\hc{A}=\cl{A}^\hcHd$ and $\hbd A= \hc{A}\sm A$. (Here, for complex notation for a subset of $\Hd$ (of the form e.g.\ $A_x^y(z)$), we use the same notational convention for $\,\hc\cdot\,$ as for $\,\c\cdot\,$---see Remark \ref{smarthat}.)
\par Although sometimes we use the linear and Euclidean structure of $\R^d$ in $\Hd$, the default geometry on $\Hd$ is the hyperbolic one, unless indicated otherwise. On the other hand, by the \emd{Euclidean metric of the disc model} we mean the metric on $\cHd$ induced by the embedding of $\cHd$ in $\R^d$ (as a unit disc) arising from the Poincar\'e disc model of $\Hd$. Nevertheless, we are going to treat that metric as a metric on the set $\hcHd\cup\{\infty\}=\cHd$, never really considering $\Hd$ in the disc model.
\end{df}

\begin{df}
For $k>0$ and $x\in\R^{d-1}\times\{0\}$, by $y\mapsto k\cdot y$ and $y\mapsto y+x$ (or $k\cdot$, $\cdot+x$, respectively, for short) we mean always just a scaling and a translation of $\R^d$, respectively, often as isometries of $\Hd$. (Note that restricted to $\Hd$ they are indeed hyperbolic isometries.)
\end{df}

\begin{denot}
Let $\Isom(\Hd)$ denote the isometry group of $\Hd$.
\par For any $h\in(0;1]$ and $R\in O(d)$ (the orthogonal linear group of $\R^d$) the pair $(h,R)$ determines uniquely an isometry of $\Hd$ denoted by $\Phi\hR$ such that~$\Phi\hR(o)=(0,\ldots,0,h)$ and $D\Phi\hR(o)=hR$ (as an ordinary derivative of a function $\R^{d-1}\times(0,\infty)\to\R^d$).
\par Let $\G\hR$ denote $\Phi\hR[\G]$. Similarly, for any $\Phi\in\Isom(\Hd)$ let $\G^\Phi=\Phi[\G]$. Further, in the same fashion, let $o\hR=\Phi\hR(o)$ (which is $h\cdot o$) and $o^\Phi=\Phi(o)$.
\end{denot}
\begin{df}
For any $p\in[0;1]$, whenever we consider $p$-Bernoulli bond percolation{} on $\G^\Phi$ for $\Phi\in\Isom(\Hd)$, we just take $\Phi[\omega]$, where $\omega$ denotes the random configuration in $p$-Bernoulli bond percolation\ on $G$.
\end{df}
\begin{rem}\label{remcoupl}
One can say that this is a way of coupling of the Bernoulli bond percolation\ processes on $\G^\Phi$ for $\Phi\in\Isom(\Hd)$.
\par Formally, the notion of ``$p$-Bernoulli bond percolation\ on $\G^\Phi$'' is not well-defined because for different isometries $\Phi_1$, $\Phi_2$ of $\Hd$ such that $\G^{\Phi_1}=\G^{\Phi_2}$, still the processes $\Phi_1[\omega]$ and $\Phi_2[\omega]$ are different. Thus, we are going to use the convention that the isometry $\Phi$ used to determine the process $\Phi[\omega]$ is the same as used in the notation $\G^\Phi$ determining the underlying graph.
\end{rem}
\begin{denot}
Let $L^h=\R^{d-1}\times(0;h]\subseteq\Hd$ and put $L=L^1$. (In other words, $L^h$ is the complement of some open horoball in $\Hd$, which viewed in the Poincar\'e disc model $\mathbb{B}^d$ is tangent to $\bdi\mathbb{B}^d$ at the point corresponding to $\infty$.)
\end{denot}

\begin{df}
Consider any closed set $A\subseteq\Hd$ intersecting each geodesic line only in finitely many intervals and half-lines of that line (every set from the algebra of sets generated by convex sets satisfies this condition, e.g.~$A=L^h$). Then, by $\G^\Phi\cap A$ we mean an embedded graph in $A$ with the set of vertices consisting of $V(\G^\Phi)\cap A$ and the points of intersection of the edges of $\G^\Phi$ with $\bdt A$ and with the edges being all the non-degenerate components of intersections of edges of $\G^\Phi$ with $A$. The percolation process on $\G^\Phi\cap A$ considered in this paper\ is, by default, the process $\Phi[\omega]\cap A$. The same convention as in Remark \ref{remcoupl} is used for these processes.
\end{df}

\begin{rem}
To prove the main theorem, we use the process $\Phi\hR[\omega]\cap L^H$ for $p\in[0;1]$ and for different $H$. In some sense, it is $p$-Bernoulli bond percolation{} on $\G\hR\cap L^H$: on one hand, this process is defined in terms of the independent random states of the edges of $\G\hR$, but on the other hand, some different edges of the graph $\G\hR\cap L^H$ are obtained from the same edge of $\G\hR$, so their states are stochastically dependent. Nevertheless, we are going to use some facts about Bernoulli percolation\ for the percolation process on $\G\hR\cap L^H$. In such situation, we consider the edges of $\G\hR$ intersecting $L^H$ instead of their fragments obtained in the intersection with $L^H$.
\end{rem}

\section{Exponential decay of the cluster size distribution}\label{secexpdec}

We are going to treat the percolation process $\Phi\hR[\omega]\cap L^H$ roughly as a Bernoulli percolation process on the standard lattice $\mathbb{Z}^{d-1}$ (given graph structure by joining every pair of vertices from $\mathbb{Z}^{d-1}$ with distance $1$ by an egde). It is motivated by the fact that $\mathbb{Z}^{d-1}$ with the graph metric is quasi-isometric to $\hbd\Hd$ or $L^H$ with the Euclidean metric. (Two metric spaces are \emd{quasi-isometric} if, loosely speakig, there are mappings in both directions between them which are bi-Lipschitz up to an additive constant. For strict definition, see \cite[Definition I.8.14]{BH}; cf.\ also Exercise 8.16(1) there.)

In the setting of $\mathbb{Z}^{d-1}$, we have a theorem on exponential decay of the cluster size distribution, below the critical threshold of percolation:
\begin{thm}[\protect{\cite[Theorem (5.4)]{Grim}}]\label{thm5.4}
For any $p<\pc(\mathbb{Z}^d)$ there exists $\psi(p)>0$ such that in $p$-Bernoulli bond percolation{} on $\mathbb{Z}^d$
$$\Pr_p(\textrm{the origin $(0,\ldots,0)$ is connected to the sphere of radius }n)<e^{-\psi(p)n}\quad\textrm{for all }n,$$
where the spheres are considered in the graph metric on $\mathbb{Z}^d$.
\\\qed
\end{thm}
The idea (of a bit more general theorem) comes from \cite{Men}, where a sketch of proof is given, and a detailed proof of the above statement is present in \cite{Grim}.

We adapt the idea of this theorem to the percolation process on $\G\hR\cap L$ in Theorem \ref{lemMen} and Lemma \ref{corMen}, appropriately rewriting the proof in \cite{Grim}, which is going to be the key part of the proof of main theorem.
In order to consider such counterpart of the above theorem, we define a kind of tail of all the distributions of the cluster size in $\G\hR\cap L$ for $(h,R)\in\IOd$ as follows:

\begin{df}\label{gpr}
Let $\pbd$ be the Euclidean orthogonal projection from $\Hd$ onto $\hbd\Hd$ and for any $x,y\in\Hd$,
$$d_\hbd(x,y)=\|\pbd(x)-\pbd(y)\|_\infty,$$
where $\|\cdot\|_\infty$ is the maximum (i.e.\ $l^\infty$) norm on $\hbd\Hd=\R^{d-1}\times\{0\}$.
Then, for $r>0$ and $x\in\Hd$, let
$$B_r(x)=\{y\in\Hd:d_\hbd(x,y)\le r\}\quad\textrm{and}\quad S_r(x)=\bdt B_r(x)$$
and for $h>0$, put
$$B_r^h(x)=B_r(x)\cap L^h.$$
If $x=o$ (or, more generally, if $\pbd(x)=\pbd(o)$), then we omit ``$(x)$''.
At last, for $p\in[0;1]$ and $r>0$, let
$$g_p(r)=\sup_{(h,R)\in\IOd}\Pr_p( \ho\conn S_r\textrm{ in }\G\hR\cap L).$$
\end{df}
\begin{rem}
In the Euclidean geometry, $B_r(x)$ and $B_r^h(x)$ are just cuboids of dimensions $r\times\ldots\times r\times\infty$ (unbounded in the direction of $d$-th axis) and $r\times\ldots\times r\times h$, respectively (up to removal of the face lying in $\hbd\Hd$).
\end{rem}
The condition ``$p<p_c(\mathbb{Z}^d)$'' in Theorem \ref{thm5.4} is going to be replaced by ``$p<p_0$'', which is natural because of the remark below. Before making it, we introduce notation concerning the percolation clusters:
\begin{denot}\label{dfMhRL}
For $\Phi\in\Isom(\Hd)$ and $v\in V(\G^\Phi)$ and a set $A\subseteq\Hd$ from the algebra generated by the convex sets, let $C^\Phi(v)$ and $C_A^\Phi(v)$ be the clusters of $v$ in $\G^\Phi$ and $\G^\Phi\cap A$, respectively, in the percolation configuration. Similarly, for $(h,R)\in\IOd$ and $\Phi=\Phi\hR$, we use notations $C\hR(v)$ and $C\hR_A(v)$, respectively.
\par If $v=\Phi(o)$, we omit ``$(v)$'' for short.
\end{denot}
\begin{rem}\label{rembdd}
If $p\in\NP$, then for any $\Phi\in\Isom(\Hd)$, the cluster $C^\Phi$ is $\Pr_p$-a.s.~bounded in the Euclidean metric. The reason is as follows. Take any $p\in\NP$ and $\Phi\in\Isom(\Hd)$. Then, for any $x\in\bdi\Hd$, we have $x\notin \c C(o)$ $\Pr_p$-a.s.~as well as $x\notin \c C^\Phi$ $\Pr_p$-a.s.
If we choose $x=\infty$ (for our half-space model of $\Hd$), then $\c C^\Phi$ is $\Pr_p$-a.s.~a compact set in $\hcHd$, so $C^\Phi$ is bounded in the Euclidean metric.
\end{rem}

Now, we formulate the theorem which is the counterpart of Theorem \ref{thm5.4}. Its proof (based on that of \cite[Theorem (5.4)]{Grim}) is deferred to Section \ref{prflemMen}.
\begin{thm}[exponential decay of $g_p(\cdot)$]\label{lemMen}\label{LEMMEN}
Let a graph $G$ embedded in $\Hd$ be connected, locally finite, transitive under isometries and let its edges be geodesic segments (as in Assumption \ref{assmG}). Then, for any $p<p_0$, there exists $\psi=\psi(p)>0$ such that~for any $r>0$,
$$g_p(r)\le e^{-\psi r}.$$
\end{thm}

The next lemma is a stronger version of the above one, where we take the union of all the clusters meeting some $\Bo$ instead of the cluster of $\ho$ in $\G\hR\cap L$. In other words, here the role of $\ho$ played in Theorem \ref{lemMen} is taken over by its thickened version $\Bo\cap V(\G^\Phi)$ for any $\Phi\in\Isom(\Hd)$. That leads to the following notation:
\begin{denot1}
We put
$$\bo_\Phi=\Bo\cap V(\G^\Phi)$$
for $\Phi\in\Isom(\Hd)$.
\end{denot1}

\begin{df}
For any $C\subseteq\Hd$, we define its \emd{size} by
$$r(C)=\sup_{x\in C}d_\hbd(o,x).$$
\end{df}

\begin{lem}\label{corMen}
Let a graph $G$ embedded in $\Hd$ be connected, locally finite, transitive under isometries and let its edges be geodesic segments (as in Assumption \ref{assmG}). Then, for any $p$ such that~the conclusion of Theorem \ref{lemMen} holds (in particular, for $p<p_0$) and for any $r_0>0$, there exist $\al=\al(p,r_0),\phi=\phi(p,r_0)>0$ such that~for any $r\ge 0$,
\begin{equation}\label{ineqcorMen}
\sup_{\Phi\in\Isom(\Hd)} \Pr_p(r(\bigcup_{v\in\bo_\Phi} C_L^\Phi(v))\ge r) \le \al e^{-\phi r}.
\end{equation}
\end{lem}
\begin{proof}
First, note that it is sufficient to prove the inequality
\begin{equation}
\label{ineqcorMen<->}
\sup_{\Phi\in\Isom(\Hd)} \Pr_p(\bo_\Phi \conn S_r\textrm{ in }\G^\Phi\cap L) \le \al e^{-\phi r}
\end{equation}
for $r$ greater than some fixed $r_1>0$ in place of \eqref{ineqcorMen}. Indeed, suppose there exist $\al,\phi>0$ such that \eqref{ineqcorMen<->} holds for all $r>r_1$. We then have:
\begin{itemize}
\item for any $r>r_1$ and $\e\in(0;\min(r-r_1,1))$,
\begin{align}
\sup_{\Phi\in\Isom(\Hd)} \Pr_p(r(\bigcup_{v\in\bo} C_L^\Phi(v))\ge r) &\le \sup_{\Phi\in\Isom(\Hd)} \Pr_p(\bo_\Phi \conn S_{r-\e}\textrm{ in }\G^\Phi\cap L) \le\\
&\le \al e^{-\phi(r-\e)} \le (\al e^{\phi})e^{-\phi r},
\end{align}
\item for $r\le r_1$, the left-hand side of \eqref{ineqcorMen} is less than or equal to $1 \le e^{\phi r_1}e^{-\phi r}$.
\end{itemize}
So then we will get the lemma for any $r\ge 0$ with $\max(e^{\phi r_1},\al e^{\phi})$ put in place of $\al$.

\par Now, we prove \eqref{ineqcorMen<->} (we pick $r_1$ as above later): let $r>r_0$ and $\Phi\in\Isom(\Hd)$. The task is to pick appropriate values of $\al$ and $\phi$ independently of $r$ and $\Phi$.
\begin{df}
Put $\bo=\bo_\Phi$. For $x\in\Hd\subseteq\R^d$, let $h(x)$ denote the $d$-th coordinate of $x$ (or: Euclidean distance from $x$ to $\hbd\Hd$), which we call \emd{height} of $x$.
\end{df}
Assume for a while that $\bo\conn S_r$ in $\G^\Phi\cap L$ (note that this event may have probability $0$, e.g.~when $\bo=\emptyset$).
Consider all open paths 
in $\G^\Phi\cap L$ joining $\bo$ to $S_r$ and consider all the vertices of $G\Phi$ visited by those paths, lying in $B_r^1$. There is a non-zero finite number of vertices of maximal height among them because $\G^\Phi$ is locally finite. Choose one of these vertices at random and call it $\vh$. This $\Hd$-valued random variable is defined whenever $\bo\conn S_r$ in $\G^\Phi\cap L$.)
\begin{obs}
There exists $H\ge 1$ such that a.s. if $\vh$ is defined, then $\vh \conn S_\frac{r-r_0}{2}(\vh)$ in $\G^\Phi\cap L^{Hh(\vh)}$.
\end{obs}
\begin{proof}
 Assume that $\vh$ is defined and take a path $P$ joining $\bo$ to $S_r$ passing through $\vh$. Hyperbolic lengths of edges in $\G^\Phi$ are bounded from above (by the transitivity of $\G^\Phi$ under isometries).
That implies that for any edge of $\G^\Phi$ the ratio between the heights of any two of its points is also bounded from above by some constant $H\ge 1$ (it is going to be the $H$ in the observation). (The reason for that are the two following basic properties of the half-space model of $\Hd$:
\begin{itemize}
\item The heights of points of any fixed hyperbolic ball (of finite radius) are bounded from above and from below by some positive constants.
\item Any hyperbolic ball can be mapped onto any other hyperbolic ball of the same radius by a translation by vector from $\R^{d-1}\times\{0\}$ composed with a linear scaling of $\R^d$.
\end{itemize}
That implies that the path $P\subseteq L^{Hh(\vh)}$.
\par Now, because $P$ contains some points $x\in\Bo$ and $y\in S_r$ and, by triangle inequality, $d_\hbd(x,y)\ge r-r_0$, it follows that $d_\hbd(\vh,x)$ or $d_\hbd(\vh,y)$ is at least $\frac{r-r_0}{2}$ (again by triangle inequality). Hence, $P$ intersects $S_\frac{r-r_0}{2}(\vh)$, which finishes the proof.
\end{proof}
Based on that observation, we estimate:
\begin{align}
\Pr_p(\bo\conn S_r\textrm{ in }\G^\Phi\cap L) \label{Pbo<->Sr}\tag{$\ast$}
&= \quad\voidindop{\sum}{v\in V(\G^\Phi)\cap B_r^1}\quad \Pr_p(\vh\textrm{ is defined and }\vh=v) \le\\
&\le \quad\voidindop{\sum}{v\in V(\G^\Phi)\cap B_r^1}\quad \Pr_p(v\conn S_\frac{r-r_0}{2}(v)\textrm{ in }\G^\Phi\cap L^{Hh(v)}) =\\
&= \quad\voidindop{\sum}{v\in V(\G^\Phi)\cap B_r^1}\quad \Pr_p\Bigl(\textstyle\frac{1}{H}\cdot o\conn S_\frac{r-r_0}{2Hh(v)}\left(\frac{1}{H}\cdot o\right)\textrm{ in }
\frac{1}{Hh(v)}(\G^{\Phi} - \pbd(v))\cap L^1\Bigr),
\end{align}
by mapping the situation via the (hyperbolic) isometry $\frac{1}{Hh(v)}(\cdot - \pbd(v))$ for each $v$.
Note that because for $v\in V(\G^\Phi)\cap B_r^1$, $\frac{1}{H}\cdot o$ indeed is a vertex of $\frac{1}{Hh(v)}(\G^{\Phi} - \pbd(v))$, by the transitivity of $\G$ under isometries, we can replace the isometry $\frac{1}{Hh(v)}\(\cdot-\pbd(v)\)$ with an isometry giving the same image of $\G$ and mapping $o$ to $\frac{1}{H}\cdot o$, hence of the form $\Phi^{(1/H,R)}$. That, combined with the assumption on $p$ (the conclusion of Theorem \ref{lemMen}), gives
\begin{equation}\label{Pbo<->Sr<exp}
\eqref{Pbo<->Sr} \le \sum_{v\in V(\G^\Phi)\cap B_r^1} g_p\(\frac{r-r_0}{2Hh(v)}\)
\le \sum_{v\in V(\G^\Phi)\cap B_r^1} e^{-\psi\frac{r-r_0}{2Hh(v)}},
\end{equation}
where $\psi$ is as in Theorem \ref{lemMen}.
\par Because $B_r^1 = [-r;r]^{d-1}\times(0;1]$, one can cover it by $\left\lceil\frac{r}{r_0}\right\rceil^{d-1}$ translations of $\Bo$ by vectors from $\R^{d-1}\times\{0\}$. So, let $\{\Bo(x_i):i=1,\ldots,\left\lceil\frac{r}{r_0}\right\rceil^{d-1}\}$ be such covering. Moreover, each $\Bo(x_i)$ can be tesselated by infinitely many isometric (in the hyperbolic sense) copies of $K=\Bo\sm L^{\frac{1}{2}}$, more precisely, by: a translation of $K$, $2^{d-1}$ translations of $\frac{1}{2}K$,  $(2^{d-1})^2$ translations of $\frac{1}{2^2}K$, etc., all along $\R^{d-1}\times\{0\}$. Let $U=\sup_{\phi\in\Isom(\Hd)}\#(V(\G^\Phi)\cap\phi[K])$ ($U<\infty$ by local finiteness of $\G$). Then, splitting the sum from \eqref{Pbo<->Sr<exp} according to those tesselations,
\begin{align}
\eqref{Pbo<->Sr} &\le \sum_{i=1}^{\left\lceil\frac{r}{r_0}\right\rceil^{d-1}} \sum_{v\in V(\G^\Phi)\cap B_{r_0}^1(x_i)} e^{-\psi\frac{r-r_0}{2Hh(v)}} \le\\
&\le \left\lceil\frac{r}{r_0}\right\rceil^{d-1} \sum_{k=0}^\infty (2^{d-1})^k U\sup_{h\in[\frac{1}{2^{k+1}};\frac{1}{2^k}]} e^{-\psi\frac{r-r_0}{2Hh}} \le\\
&\le U\left\lceil\frac{r}{r_0}\right\rceil^{d-1} \sum_{k=0}^\infty (2^{d-1})^k e^{-\frac{\psi}{H}2^{k-1}(r-r_0)} =\label{smallinfsum}\\
&= U\left\lceil\frac{r}{r_0}\right\rceil^{d-1} \sum_{k=0}^\infty e^{\ln 2\cdot k(d-1) - \frac{\psi}{H}2^{k-1}(r-r_0)}.
\end{align}
Now, we are going to show that the above bound is finite and tends to $0$ at exponential rate with $r\to\infty$. First, we claim that there exists $k_0\in\N$ such that
\begin{equation}\label{k_0cond}
(\forall k\ge k_0)(\forall r\ge 2r_0) \left(\ln2\cdot k(d-1) - 2^{k-1} \frac{\psi}{H}(r-r_0) \le -kr\right).
\end{equation}
Indeed, for sufficiently large $k$ we have $2^{k-1}\frac{\psi}{H}-k>0$, so for $r\ge 2r_0$
\begin{equation}
\left(2^{k-1}\frac{\psi}{H}-k\right)r \ge \left(2^{k-1}\frac{\psi}{H}-k\right)\cdot2r_0
\end{equation}
and
\begin{equation}
2^{k-1}\frac{\psi}{H}(r-r_0) - kr \ge 2^{k-1}\frac{\psi}{H}r_0 - 2kr_0 \ge k(d-1)\ln2
\end{equation}
for sufficiently large $k$. So, let $k_0$ satisfy \eqref{k_0cond}. Then, for $r\ge 2r_0$,
\begin{align}
\eqref{Pbo<->Sr} &\le U\left\lceil\frac{r}{r_0}\right\rceil^{d-1} \(\sum_{k=0}^{k_0-1} (2^{d-1})^k e^{-2^{k-1} \frac{\psi}{H}(r-r_0)} + \sum_{k=k_0}^\infty e^{-kr}\) \le\\
&\le U\left\lceil\frac{r}{r_0}\right\rceil^{d-1} \biggl(k_0(2^{d-1})^{k_0-1} e^{-\frac{\psi}{2H}(r-r_0)} + e^{-k_0r}\underbrace{\frac{1}{1-e^{-r}}}_{\le \frac{1}{1-e^{-2r_0}}}\biggr) \le\\
&\le U\left\lceil\frac{r}{r_0}\right\rceil^{d-1} (De^{-Er})
\end{align}
for some constants $D,E>0$. If we choose $r_1\ge 2r_0$ such that
$$(\forall r\ge r_1)\left( \left\lceil\frac{r}{r_0}\right\rceil^{d-1} \le e^\frac{Er}{2}\right)$$
(which is possible), then
$$\eqref{Pbo<->Sr}\le UDe^\frac{-Er}{2}\quad\textrm{for }r\ge r_1,$$
which finishes the proof of the lemma.
\end{proof}

\section{Scaling---proof of the main theorem}\label{secprfmainthm}

Now we complete the proof of the main theorem:
\begin{thm*}[recalled Theorem \ref{mainthm}]
Let $G$ satisfy the Assumption \ref{assmG}. Then, for any $0\le p<p_0$, a.s. every cluster in $p$-Bernoulli bond percolation\ on $\G$ has only one-point boundaries of ends.
\end{thm*}
\begin{proof}[Proof of Theorem \ref{mainthm}]
Fix $p\in[0;p_0)$ and suppose towards a contradiction that with some positive probability there is some cluster with some end with a non-one-point boundary. Note that by Remark \ref{rembdd} and by the transitivity of $\G$ under isometries, for any $v\in V(\G)$ a.s.~$C(v)$ is bounded in the Euclidean metric, so, a.s.~all the percolation clusters in $\G$ are bounded in the Euclidean metric. Then, for some $\d>0$ and $r>0$, there exists with probability\ $a>0$ a cluster bounded in the Euclidean metric, with boundary of some end having Euclidean diameter greater than or equal to $\d$ and intersecting the open disc $\sint_{\hbd\Hd}\hbd B_r$.
Let $C$ and $e$ be such cluster and its end, respectively. Let for $A\subseteq\Hd$, the \emd{projection diameter} of $A$ be the Euclidean diameter of $\pbd(A)$. Then for $h>0$
\begin{itemize}
\item the set $\cl{C\sm L^h}$ is compact;
\item $e(\cl{C\sm L^h})$ is a cluster in the percolation configuration on $\G\cap L^h)$;
\item $e(\cl{C\sm L^h})$ has projection diameter at least $\d$ and intersects $B_r\cap V(\G)$.
\end{itemize}
All the above implies that for any $k\in\N$,
$$\Pr_p(\exists\textrm{ a cluster in $\G\cap L^\frac{1}{2^k}$ of projection diameter $\ge\d$ intersecting $B_r\cap V(\G)$})\ge a,$$
so, by scaling by $2^k$ in $\R^d$ (which is a hyperbolic isometry), we obtain
$$\Pr_p(\exists\textrm{ a cluster in $\G^{2^k\cdot}\cap L$ of projection diameter $\ge2^k\d$ intersecting $B_{2^kr}\cap V(\G^{2^k\cdot})$})\ge a$$
(where $\G^{2^k\cdot}$ is the image of $\G$ under the scaling). $B_{2^kr}\cap L$ is a union of $(2^k)^{d-1}$ isometric copies of $B_r\cap L$, so the left-hand side of above inequality is bounded from above by
\begin{multline}
(2^k)^{d-1}\sup_{\Phi\in\Isom(\Hd)} \Pr_p(\exists\textrm{ a cluster in $\G^\Phi\cap L$ of proj.\ diam.\ $\ge2^k\d$ intersecting $B_{r}\cap V(\G^\Phi)$}) \le\\
\le (2^k)^{d-1}\sup_{\Phi\in\Isom(\Hd)} \Pr_p\left(r\left(\bigcup_{v\in B_r^1\cap V(\G^\Phi)} C_L^\Phi(v)\right)\ge \frac{2^k\d}{2}\right)
\end{multline}
(because the size of a cluster is at least half its projection diameter),
so by Lemma \ref{corMen}, for any $k\in\N$,
$$a\le (2^k)^{d-1} \al e^{-\phi\d 2^{k-1}},$$
where $\al,\phi>0$ are constants (as well as $\d$, $a$ and $r$). But the right-hand side of this inequality tends to $0$ with $k\to \infty$, a contradiction.
\end{proof}

\section{Proof of the exponential decay}\label{prflemMen}
In this section, we prove Theorem \ref{lemMen}:
\begin{thm*}[recalled Theorem \ref{lemMen}]
Let a graph $G$ embedded in $\Hd$ be connected, locally finite, transitive under isometries and let its edges be geodesic segments (as in Assumption \ref{assmG}). Then, for any $p<p_0$, there exists $\psi=\psi(p)>0$ such that~for any $r>0$,
$$g_p(r)\le e^{-\psi r}.$$
\end{thm*}
\begin{proof}
As mentioned earlier, this proof is an adaptation of the proof of Theorem (5.4) in \cite{Grim} based on the work \cite{Men}. Its structure and most of its notation are also borrowed from \cite{Grim}, so it is quite easy to compare both the proofs. (The differences are technical and they are summarised in Remark \ref{diffprfs}.) The longest part of this proof is devoted to show functional inequality \eqref{fctineq(gt)1} and it goes roughly linearly. Then follows Lemma \ref{gt<=sqrt}, whose proof, using that functional inequality, is deferred to the end of this section. Roughly speaking, that lemma provides a mild asymptotic estimate for $g_p$ (more precisely: for $\gt_p$ defined below), which is then sharpened to that desired in Theorem \ref{lemMen}, using repeatedly inequality \eqref{fctineq(gt)1}.
\par At some point, we would like to use random variables with left-continuous distribution function\footnote{By the \emd{left-continuous distribution function} of a probability distribution (measure) $\mu$ on $\R$, we mean the function $\R\ni x\mapsto\mu((-\infty,x))$.} $1-g_p$. Because $1-g_p$ does not need to be left-continuous, we replace $g_p$, when needed, by its left-continuous version $\gt_p$ defined as follows:
\begin{df}
Put $\gt_p(r)=\lim_{\rho\to r^-} g_p(\rho)$ for $r>0$.
\end{df}
As one of the cornerstones of this proof, we are going to prove the following functional inequality for $\gt_\cdot(\cdot)$: for any $\al,\b$ s.t.~$0\le\al<\b\le 1$ and for $r>0$,
\begin{equation}\label{fctineq(gt)1}
\gt_\al(r) \le \gt_\b(r)\exp\(-(\b-\al)\(\frac{r}{a+\int_0^r \gt_\b(m)\,\ud m} - 1\)\),
\end{equation}
where $a$ is a positive constant depending only on $G$. Note that it implies Theorem \ref{lemMen} provided that the integral in the denominator is a bounded function of $r$.
\par We are going to approach this inequality, considering the following events depending only on a finite fragment of the percolation configuration and proving functional inequality \eqref{fctineq(f)} (see below). Cf.\ Remark \ref{whyLd}.
\begin{df}
Fix arbitrary $(h,R)\in\IOd$. We are going to use events $A^\d(r)$ defined as follows: let $p\in[0,1]$, $r>0$ and $\d\in(0;h]$, and define $L_\d=\R^{d-1}\times[\d;1]\subseteq\Hd$ (not to be confused with $L^\d$). Let the event
$$A^\d(r)=\{\ho\conn S_r\textrm{ in }\G\hR\cap L_\d\}$$
and let
$$f_p^\d(r)=\Pr_p(A^\d(r)).$$
\end{df}
Now, we are going to show that the functions defined above satisfy a functional inequality:
\begin{equation}\label{fctineq(f)}
f_\al^\d(r)\le f_\b^\d(r)\exp\(-(\b-\al)\(\frac{r}{a+\int_0^r \gt_\b(m)\,\ud m} - 1\)\)
\end{equation}
for any $0\le\al<\b\le 1$, $r>0$ and for $\d\in(0;h)$.
Having obtained this, we will pass to some limits and to supremum over $(h,R)$, obtaining the inequality \eqref{fctineq(gt)1}.
\par Note that, if there is no path joining $\ho$ to $S_r$ in $\G\hR\cap L_\d$ at all, then for any $p\in[0;1]$, $f_p^\d(r)=0$ and the inequality \eqref{fctineq(f)} is obvious. The same happens when $\al=0$. Because in the proof of that inequality we need $f_p^\d(r)>0$ and $\al>0$, now we make the following assumption (without loss of generality):
\begin{assm}\label{assmapath}
We assume that there is a path joining $\ho$ to $S_r$ in $\G\hR\cap L_\d$ and that $\al>0$. (Then for $p>0$, $f_p^\d(r)>0)$.)
\end{assm}
The first step in proving inequality \eqref{fctineq(f)} in using Russo's formula for the events $A^\d(r)$. Before we formulate it, we provide a couple of definition needed there.
\begin{df}
Next, for a random event $A$ in the percolation on $\G\hR$, call an edge \emd{pivotal} for a given configuration iff changing the state of that edge (and preserving other edges' states) causes $A$ to change its state as well (from occurring to not occurring or \emph{vice-versa}). Then, let $N(A)$ be the (random) number of all edges pivotal for $A$.
\end{df}
\begin{df}\label{incr}
We say that an event $A$ (being a set of configurations) is \emd{increasing} iff for any configurations $\omega\subseteq\omega'$, if $\omega\in A$, then $\omega'\in A$.
\end{df}
\begin{thm}[Russo's formula]
Consider Bernoulli bond percolation\ on any graph $G$ and let $A$ be an increasing event defined in terms of the states of only finitely many edges of $G$. Then
$$\frac{\ud}{\ud p}\Pr_p(A)=\Est_p(N(A)).$$
\qed
\end{thm}
This formula is proved as Theorem (2.25) in \cite{Grim} for $G$ being the classical lattice $\mathbb{Z}^d$, but the proof applies for any graph $G$.
\par Let $p>0$. The events $A^\d(r)$ depend on the states of only finitely many edges of $\G\hR$ (namely, those intersecting $L_\d\cap B_r$), so we are able to use Russo's formula for them, obtaining
$$\frac{\ud}{\ud p} f_p^\d(r) = \Est_p(N(A^\d(r))).$$
Now, for $e\in E(\G\hR)$, the event $\{e\textrm{ is pivotal for }A^\d(r)\}$ is independent of the state of $e$ (which is easily seen; it is the rule for any event), so
\begin{align}
\Pr_p(A^\d(r)\land e\textrm{ is pivotal for }A^\d(r)) &= \Pr_p(e\textrm{ is open and pivotal for }A^\d(r)) =\\
&=\textrm{(because $A^\d(r)$ is increasing)}\\
&= p\Pr_p(e\textrm{ is pivotal for }A^\d(r)),
\end{align}
hence
\begin{align}
\frac{\ud}{\ud p} f_p^\d(r) &= \sum_{e\in E(\G\hR)} \Pr_p(e\textrm{ is pivotal for }A^\d(r)) =\\
&= \frac 1p \sum_{e\in E(\G\hR)} \Pr_p(A^\d(r)\land e\textrm{ is pivotal for }A^\d(r)) =\\
&=\frac{f_p^\d(r)}{p} \sum_{e\in E(\G\hR)} \Pr_p(e\textrm{ is pivotal for }A^\d(r)|A^\d(r)) =\\
&= \frac{f_p^\d(r)}{p} \Est_p(N(A^\d(r))|A^\d(r)),
\end{align}
which can be written as
$$\frac{\ud}{\ud p} \ln(f_p^\d(r)) = \frac{1}{p} \Est_p(N(A^\d(r))|A^\d(r)).$$
For any $0<\al<\b\le 1$, integrating over $[\al,\b]$ and exponentiating the above equality gives
$$\frac{f_\al^\d(r)}{f_\b^\d(r)} = \exp\(-\int_\al^\b \frac 1p \Est_p(N(A^\d(r))|A^\d(r))\,\ud p\),$$
which implies
\begin{align}
f_\al^\d(r) 
\le f_\b^\d(r) \exp\(-\int_\al^\b \Est_p(N(A^\d(r))|A^\d(r))\,\ud p\).\label{corineqRusso}
\end{align}
At this point, our aim is to bound $\Est_p(N(A^\d(r))|A^\d(r))$ from below.
\begin{df}
Let $\eta$ denote the percolation configuration in $\G\hR\cap L_\d$, i.e.
$$\eta = \Phi\hR[\omega]\cap L_\d.$$
\end{df}
Fix any $r>0$ and $\d\in(0;h]$ and assume for a while that $A^\d(r)$ occurs. Let us make a picture of the cluster of $\ho$ in $\G\hR\cap L_\d$ in the context of the pivotal edges for $A^\d(r)$ (the same picture as in \cite{Grim} and \cite{Men}). If $e\in E(\G\hR)$ is pivotal for $A^\d(r)$, then if we change the percolation configuration by closing $e$, we cause the cluster of $C\hR_{L_\d}(\ho)$ to be disjoint from $S_r$. So, in our situation, all the pivotal edges lie on any open path in $\G\hR\cap L_\d$ joining $\ho$ to $S_r$ and they are visited by the path in the same order and direction (regardless of the choice of the path).
\begin{df}\label{dfe_i}
Let $N=N(A^\d(r))$, and let $e_1,\ldots,e_N$ be this ordering, and denote by $x_i,y_i$ the endvertices of $e_i$, $x_i$ being the one closer to $\ho$ along a path as above. Also, let $y_0=\ho$.
\end{df}
Note that because for $i=1,\ldots,N$, there is no edge separating $y_{i-1}$ from $x_i$ in the open cluster in $\G\hR\cap L_\d$, by Menger's theorem (see e.g.~\cite[Theorem~3.3.1, Corollary~3.3.5(ii)]{Diest}), there exist two edge-disjoint open paths in that cluster joining $y_{i-1}$ to $x_i$. (One can say, following the discoverer of this proof idea, that that open cluster resembles a chain of sausages.)
\begin{df}
Now, for $i=1,\ldots,N$, let $\rho_i=d_\hbd(y_{i-1},x_i)$ (this way of defining $\rho_i$, of which one can think as ``projection length'' of the $i$-th ``sausage'', is an adaptation of that from \cite{Grim}).
\end{df}
\par Now we drop the assumption that $A^\d(r)$ occurs. The next lemma is used to compare $(\rho_1,\ldots,\rho_N)$ to some renewal process with inter-renewal times of roughly the same distribution as the size of $C_L\hR(\ho)$.
\begin{df}
Let $a$ denote the maximal projection distance (in the sense of $d_\hbd$) between the endpoints of a single edge $\G\hR$ crossing $L$.
\end{df}
\begin{lem}[cf.~\protect{\cite[Lemma~(5.12)]{Grim}}]\label{saus}
Let $k\in\N_+$ and let $r_1,\ldots,r_k\ge 0$ be such that~$\sum_{i=1}^k r_i \le r - (k-1)a$. Then for $0<p<1$,
$$\Pr_p(\rho_k<r_k,\ \rho_i=r_i\textrm{ for }i<k|A^\d(r)) \ge (1-g_p(r_k))\Pr_p(\rho_i=r_i\textrm{ for }i<k|A^\d(r)).$$
\end{lem}
\begin{rem}
We use the convention that for $i\in\N$ such that~$i>N(A^\d(r))$ (i.e.~$e_i$, $\rho_i$ are undefined), $\rho_i=+\infty$ (being greater than any real number). On the other hand, whenever we mention $e_i$, $i\le N(A^\d(r))$.
\end{rem}
\begin{proof}[Proof of the lemma]
This proof mimics that of \cite[Lemma~(5.12)]{Grim}. Let $k\ge 2$ (we defer the case of $k=1$ to the end of the proof).
\begin{df}
Let for $e\in E(\G\hR\cap L^\d)$, $D_e$ be the connected component of $\ho$ in $\eta\sm\{e\}$. Let $B_e$ denote the event that the following conditions are satisfied:
\begin{itemize}
\item $e$ is open;
\item exactly one endvertex of $e$ lies in $D_e$---call it $x(e)$ and the other---$y(e)$;
\item $D_e$ is disjoint from $S_r$;
\item there are $k-1$ pivotal edges for the event $\{\ho\conn y(e)\textrm{ in }\eta\}$ (i.e. the edges each of which separates $\ho$ from $y(e)$ in $D_e\cup\{e\}$)---call them $e_1'=\{x_1',y_1'\},\ldots,e_{k-1}'=\{x_{k-1}',y_{k-1}'\}=e$, $x_i'$ being closer to $\ho$ than $y_i'$, in the order from $\ho$ to $y(e)$ (as in the Definition \ref{dfe_i});
\item $d_\hbd(y_{i-1}',x_i')=r_i$ for $i<k$, where $y_0'=\ho$.
\end{itemize}
Let $B=\bigcup_{e\in E(\G\hR\cap L^\d)} B_e$. When $B_e$ occurs, we say that $D_e\cup\{e\}$ with $y(e)$ marked, as a graph with distinguished vertex, is a \emd{witness for $B$}.
\end{df}
\par Note that it may happen that there are more than one such witnesses (which means that $B_e$ occurs for many different $e$). On the other hand, when $A^\d(r)$ occurs, then $B_e$ occurs for only one edge $e$, namely $e=e_{k-1}$ (in other words, $B\cap A^\d(r)=\bigcupdot_{e\in E(\G\hR\cap L^\d)} (B_e\cap A^\d(r))$), and there is only one witness for $B$. Hence,
$$\Pr_p(A^\d(r)\cap B) = \sum_{\Gamma} \Pr_p(\Gamma\textrm{ a witness for }B)\Pr_p(A^\d(r)|\Gamma\textrm{ a witness for }B),$$
where the sum is always over all $\Gamma$ being finite subgraphs of $\G\hR\cap L_\d$ with distinguished vertices such that~$\Pr_p(\Gamma\textrm{ a witness for }B)>0$.
\par For $\Gamma$ a graph with distinguished vertex, let $y(\Gamma)$ denote that vertex. Under the condition that $\Gamma$ is a witness for $B$, $A^\d(r)$ is equivalent the event that $y(\Gamma)$ is joined to $S_r$ by an open path in $\eta$ which is disjoint from $V(\Gamma)\sm \{y(\Gamma)\}$. We shortly write the latter event $\{y(\Gamma)\conn S_r\textrm{ in $\eta$ off }\Gamma\}$.
Now, the event $\{\Gamma\textrm{ a witness for }B\}$ depends only on the states of edges incident to vertices from $V(\Gamma)\sm\{y(\Gamma)\}$, so it is independent of the event $\{y(\Gamma)\conn S_r\textrm{ in $\eta$ off }\Gamma\}$. Hence,
\begin{equation}\label{P(AB)}
\Pr_p(A^\d(r)\cap B) = \sum_{\Gamma} \Pr_p(\Gamma\textrm{ a witness for }B)\Pr_p(y(\Gamma)\conn S_r\textrm{ in $\eta$ off }\Gamma).
\end{equation}
A similar reasoning, performed below, gives us the estimate of $\Pr_p(\{\rho_k\ge r_k\}\cap A^\d(r)\cap B)$. Here we use also the following fact: conditioned on the event $\{\Gamma\textrm{ a witness for }B\}$, the event $A^\d(r)\cap\{\rho_k\ge r_k\}$ is equivalent to each of the following:
\begin{align*}
&(A^\d(r)\land e_k\textrm{ does not exist}) \lor (A^\d(r)\land e_k\textrm{ exists }\land\rho_k\ge r_k) \iff\\
\iff &(\exists\textrm{ two edge-disjoint paths joining $y(\Gamma)$ to $S_r$ in $\eta$ off }\Gamma) \lor\\
\lor &(\exists\textrm{ two edge-disjoint paths in $\eta$ off $\Gamma$, joining $y(\Gamma)$ to $S_r$ and to $S_{r_k}(y(\Gamma))$, resp.}) \iff\\
\iff &(\exists\textrm{ two edge-disjoint paths in $\eta$ off $\Gamma$, joining $y(\Gamma)$ to $S_r$ and to $S_{r_k}(y(\Gamma))$, resp.}),
\end{align*}
because $S_{r_k}(y(\Gamma))\subseteq B_r$ from the assumption on $\sum_{i=1}^k r_i$.
So we estimate
\begin{multline}\label{P(rhoAB)}
\Pr_p(\{\rho_k\ge r_k\}\cap A^\d(r)\cap B) = \sum_{\Gamma} \Pr_p(\Gamma\textrm{ a witness for }B) \Pr_p(\{\rho_k\ge r_k\}\cap A^\d(r)|\Gamma\textrm{ a witness for }B) =\\
= \sum_{\Gamma} \Pr_p(\Gamma\textrm{ a witness for }B) \Pr_p((y(\Gamma)\conn S_r\textrm{ in $\eta$ off }\Gamma) \circ (y(\Gamma)\conn S_{r_k}(y(\Gamma))\textrm{ in $\eta$ off }\Gamma)),
\end{multline}
where the operation ``$\circ$'' is defined below:
\begin{df}
For increasing events $A$ and $B$ in a percolation on any graph $G$, the event $A\circ B$ means that ``$A$ and $B$ occur on disjoint sets of edges''. Formally,
$$A\circ B = \{\omega_A\cupdot\omega_B: \omega_A,\omega_B\subseteq E(G) \land \omega_A\in A\land\omega_B\in B\},$$
that is, $A\circ B$ is the set of configurations containing two disjoint set of open edges ($\omega_A,\omega_B$ above) which guarantee occurring of the events $A$ and $B$, respectively.
\end{df}
Now, we are going to use the following BK inequality (proved in \cite{Grim}):
\begin{thm}[BK inequality, \protect{\cite[Theorems (2.12) and (2.15)]{Grim}}]\label{BKineq}
For any graph $G$ and increasing events $A$ and $B$ depending on the states of only finitely many edges in $p$-Bernoulli bond percolation\ on $G$, we have
$$\Pr_p(A\circ B) \le \Pr_p(A)\Pr_p(B).$$
\qed
\end{thm}
We use this inequality for the last term (as the events involved are increasing (see def.~\ref{incr}) and defined in terms of only the edges from $E(\G\hR\cap(L_\d\cap B_r))$), obtaining
\begin{eqnarray*}
\Pr_p(\{\rho_k\ge r_k\}\cap A^\d(r)\cap B) &\le &\sum_{\Gamma} \Pr_p(\Gamma\textrm{ a witness for }B) \cdot \Pr_p(y(\Gamma)\conn S_r\textrm{ in $\eta$ off }\Gamma) \cdot\\
&& {}\cdot \Pr_p(y(\Gamma)\conn S_{r_k}(y(\Gamma))\textrm{ in $\eta$ off }\Gamma) \le\\
&\le &\(\sum_{\Gamma} \Pr_p(\Gamma\textrm{ a witness for }B) \Pr_p(y(\Gamma)\conn S_r\textrm{ in $\eta$ off }\Gamma)\)g_p(r_k) =\\
&= &\Pr_p(A^\d(r)\cap B) g_p(r_k)
\end{eqnarray*}
(by \eqref{P(AB)}). Dividing by $\Pr_p(A^\d(r))$ (which is positive by Assumption \ref{assmapath}) gives
\begin{align}\label{P(|A)}
\Pr_p(\{\rho_k\ge r_k\}\cap B|A^\d(r)) &\le \Pr_p(B|A^\d(r)) g_p(r_k) \quad|\quad\Pr_p(B|A^\d(r)) - \cdot\\
\Pr_p(\{\rho_k< r_k\}\cap B|A^\d(r)) &\ge \Pr_p(B|A^\d(r)) (1-g_p(r_k)).
\end{align}
Note that, conditioned on $A^\d(r)$, $B$ is equivalent to the event $\{\rho_i=r_i\textrm{ for }i<k\}$, so the above amounts to
\begin{equation}
\Pr_p(\rho_k<r_k, \rho_i=r_i\textrm{ for }i<k |A^\d(r)) \ge \Pr_p(\rho_i=r_i\textrm{ for }i<k|A^\d(r)) (1-g_p(r_k)),
\end{equation}
which is the desired conclusion.
\par Now, consider the case of $k=1$. In this case, similarly to \eqref{P(rhoAB)} and thanks to the assumption $r_1\le r$,
\begin{align}
\Pr_p(\{\rho_1\ge r_1\}\cap A^\d(r)) = \Pr_p((\ho\conn S_{r_1}\textrm{ in }\eta)\circ(\ho\conn S_r\textrm{ in }\eta)) \le
g_p(r_1) \Pr_p(A^\d(r)).
\end{align}
Further, similarly to \eqref{P(|A)},
\begin{equation}
\Pr_p(\rho_1<r_1|A^\d(r))\ge 1 - g_p(r_1),
\end{equation}
which is the lemma's conclusion for $k=1$.
\end{proof}

Now, we want to do some probabilistic reasoning using random variables with the left-continuous distribution function $1-\gt_p$.
The function $1-\gt_p$ is non-decreasing (because for $(h,R)\in\IOd$, $\Pr_p(\ho \conn S_r\textrm{ in }\G\hR\cap L)$ is non-increasing with respect to $r$, so $g_p$ and $\gt_p$ are non-increasing as well), left-continuous, with values in $[0;1]$ and such that~$1-\gt_p(0)=0$, so it is the left-continuous distribution function of a random variable with values in $[0;\infty]$.
\begin{denot}
Let $M_1, M_2,\ldots$ be an infinite sequence of independent random variables all distributed according to $1-\gt_p$ and all independent of the whole percolation process. Because their distribution depends on $p$, we will also denote them by $M_1^{(p)},M_2^{(p)},\ldots$. (Here, an abuse of notation is going to happen, as we are still writing $\Pr_p$ for the whole probability measure used also for defining the variables $M_1, M_2,\ldots$.)
\end{denot}
We can now state the following corollary of Lemma \ref{saus}:
\begin{cor}\label{corsaus}
For any $r>0$, positive integer $k$ and $0<p<1$,
$$\Pr_p(\rho_1+\cdots+\rho_k < r-(k-1)a |A^\d(r)) \ge \Pr_p(M_1+\cdots+M_k < r-(k-1)a ).$$
\end{cor}
\begin{proof}
We compose the proof of the intermediate inequalities:
\begin{align*}
&\Pr_p(\rho_1+\cdots+\rho_k < r-(k-1)a |A^\d(r)) \ge\\
\ge &\Pr_p(\rho_1+\cdots+\rho_{k-1}+M_k < r-(k-1)a |A^\d(r)) \ge\cdots\\
\cdots\ge &\Pr_p(\rho_1+M_2+\cdots+M_k < r-(k-1)a |A^\d(r)) \ge\\
\ge &\Pr_p(M_1+\cdots+M_k < r-(k-1)a |A^\d(r)) = \Pr_p(M_1+\cdots+M_k < r-(k-1)a)
\end{align*}
using the step:
\begin{align}
&\Pr_p(\rho_1+\cdots+\rho_j + M_{j+1}+\cdots+M_k < r-(k-1)a |A^\d(r)) \ge\\
\ge &\Pr_p(\rho_1+\cdots+\rho_{j-1} + M_j+\cdots+M_k < r-(k-1)a |A^\d(r)).
\end{align}
for $j=k,k-1,\ldots,2,1$. Now we prove this step: let $j\in\{1,2,\ldots k\}$.
\begin{df}
Put
$$\RhR = \{d_\hbd(x,y): x,y\in V(\G\hR)\}.$$
\end{df}
Note that it is a countable set of all possible values of $\rho_i$ for $i=1,\ldots,N$.

We express the considered probability as an integral, thinking of the whole probability space as Cartesian product of the space on which the percolation processes are defined and the space used for defining $M_1,M_2,\ldots$, and using a version of Fubini theorem for events:
\begin{align*}
&\Pr_p(\rho_1+\cdots+\rho_j + M_{j+1}+\cdots+M_k < r-(k-1)a |A^\d(r)) =\\
= &\int \Pr_p(\rho_1+\cdots+\rho_j + S_M < r-(k-1)a |A^\d(r))\,\ud\L_{j+1}^k(S_M) =\\
\intertext{(here $\L_{j+1}^k$ denotes the distribution of the random variable $M_{j+1}+\cdots+M_k$)}
= &\int \rvoidindop{\sum}{(r_1,\ldots,r_{j-1})}\quad
 \Pr_p\left(\left.\rho_i=r_i\textrm{ for }i<j\land \rho_j < r-(k-1)a - \sum_{i=1}^{j-1} r_i - S_M \right|A^\d(r)\right)\,\ud\L_{j+1}^k(S_M) \ge\\
\intertext{(where the sum is taken over all $(r_1,\ldots,r_{j-1})\in(\RhR)^{j-1}:r_1+\cdots+r_{j-1}<r-(k-1)a-S_M$)}
\ge &\int \rvoidindop{\sum}{(r_1,\ldots,r_{j-1})}\quad
 \left(1-\gt_p\left(r-(k-1)a - \sum_{i=1}^{j-1} r_i - S_M\right)\right) \Pr_p(\rho_i=r_i\textrm{ for }i<j |A^\d(r))\,\ud\L_{j+1}^k(S_M) =\\
\intertext{(from Lemma \ref{saus} and because $g_p\le\gt_p$)}
= &\int \rvoidindop{\sum}{(r_1,\ldots,r_{j-1})}\quad
 \Pr_p \left(\left.M_j < r-(k-1)a - \sum_{i=1}^{j-1} r_i - S_M \land \rho_i=r_i\textrm{ for }i<j \right|A^\d(r)\right)\,\ud\L_{j+1}^k(S_M) =\\
= &\int \Pr_p(\rho_1+\cdots+\rho_{j-1} + M_j+S_M < r-(k-1)a |A^\d(r))\,\ud\L_{j+1}^k(S_M) =\\
= &\Pr_p(\rho_1+\cdots+\rho_{j-1} + M_j+M_{j+1}+\cdots+M_k < r-(k-1)a |A^\d(r)).
\end{align*}
That completes the proof.
\end{proof}

\begin{lem}[cf.~\protect{\cite[Lemma~(5.17)]{Grim}}]\label{lemEN>=frac}
For $0<p<1$, $r>0$,
$$\Est_p(N(A^\d(r))|A^\d(r)) \ge \frac{r}{a+\int_0^r \gt_p(m)\,\ud m}-1.$$
\end{lem}
\begin{proof}
For any $k\in\N_+$, if $\rho_1+\cdots+\rho_k < r-(k-1)a$, then $e_1,\ldots,e_k$ exist and $N(A^\d(r))\ge k$. So, from the corollary above,
\begin{align}\label{P(N>=k|A)>=}
\Pr_p(N(A^\d(r))\ge k|A^\d(r)) \ge \Pr_p(\sum_{i-1}^k \rho_i < r-(k-1)a) \ge \Pr_p (\sum_{i-1}^k M_i < r-(k-1)a).
\end{align}
Now, we use a calculation which relates $a+\int_0^r \gt_p(m)\,\ud m$ to the distribution of $M_1$. Namely, we replace the variables $M_i$ by
$$M_i' = a+\min(M_i,r)$$
for $i=1,2,\ldots$ (a kind of truncated version of $M_i$). In this setting,
\begin{align}
\sum_{i=1}^k M_i < r-(k-1)a \iff
\sum_{i=1}^k \min(M_i,r) < r-(k-1)a \iff
\sum_{i=1}^k M_i' < r+a,
\end{align}
so from \eqref{P(N>=k|A)>=},
\begin{align}
\Est_p(N(A^\d(r))|A^\d(r)) &= \sum_{k=1}^\infty \Pr_p(N(A^\d(r))\ge k |A^\d(r)) \ge\\
&\ge \sum_{k=1}^\infty \Pr_p(\sum_{i=1}^k M_i' < r+a) = \sum_{k=1}^\infty \Pr_p(K\ge k+1) =\\
&= \Est(K)-1,
\end{align}
where
$$K = \min\{k:M_1'+\cdots+M_k'\ge r+a\}.$$
Let for $k\in\N$,
$$S_k=M_1'+\cdots+M_k'.$$
By Wald's equation (see e.g.\ \cite[p.~396]{GrimSti}) for the random variable $S_K$,
$$r+a\le \Est(S_K)=\Est(K)\Est(M_1').$$
In order that Wald's equation were valid for $S_K$, the random variable $K$ has to satisfy $\Est(M_i'|K\ge i) = \Est(M_i')$ for $i\in\N_+$. But we have
$$K\ge i \iff M_1'+\cdots+M_{i-1}'<r+a,$$
so $M_i'$ is independent of the event $\{K\ge i\}$ for $i\in\N_+$, which allows us to use Wald's equation. (In fact, $K$ is a so-called stopping time for the sequence $(M_i')_{i=1}^\infty$.) Hence,
\begin{align}
\Est_p(N(A^\d(r))|A^\d(r)) &\ge \Est(K)-1 \ge \frac{r+a}{\Est(M_1')}-1 =\\
&= \frac{r+a}{a+\int_0^\infty \Pr_p(\min(M_1,r)\ge m)\,\ud m} -1 \ge \frac{r}{a+\int_0^r \gt_p(m)\,\ud m}-1,
\end{align}
which finishes the proof.
\end{proof}

Now, combining that with inequality \eqref{corineqRusso} for $0<\al<\b\le 1$, we have
\begin{align}
f_\al^\d(r) &\le f_\b^\d(r) \exp\(-\int_\al^\b \Est_p(N(A^\d(r))|A^\d(r))\,\ud p\) \le\\
&\le f_\b^\d(r) \exp\(-\int_\al^\b \(\frac{r}{a+\int_0^r \gt_p(m)\,\ud m}-1\)\,\ud p\) \le\\
&\le f_\b^\d(r) \exp\(-(\b-\al) \(\frac{r}{a+\int_0^r \gt_\b(m)\,\ud m}-1\)\)\label{fctineq(f)beflims}
\end{align}
(because $\gt_p\le \gt_\b$ for $p\le\b$), which completes the proof of inequality \ref{fctineq(f)}. (Let us now drop Assumption \ref{assmapath}.)

\par Now, note that for any $r>0$ and $p\in[0;1]$, the event $A^\d(r)$ increases as $\d$ decreases. Thus, taking the limit with $\d\to 0$, we have
$$\lim_{\d\to 0^+} f_p^\d(r) = \Pr_p(\bigcup_{\d> 0} A^\d(r)) = \Pr_p(\ho\conn S_r\textrm{ in }\G\hR\cap L).$$
So for any $r>0$ and $0\le\al<\b\le 1$, using this for inequality \ref{fctineq(f)} gives
\begin{multline}
\Pr_\al(\ho\conn S_r\textrm{ in }\G\hR\cap L) \le\\
\le \Pr_\b(\ho\conn S_r\textrm{ in }\G\hR\cap L) \exp\(-(\b-\al) \(\frac{r}{a+\int_0^r \gt_\b(m)\,\ud m}-1\)\).
\end{multline}
Further, we take the supremum over $(h,R)\in\IOd$, obtaining
$$g_\al(r) \le g_\b(r) \exp\(-(\b-\al) \(\frac{r}{a+\int_0^{r} \gt_\b(m)\,\ud m}-1\)\).$$
At last, taking the limits with $r$ from the left, we get the functional inequality \eqref{fctineq(gt)1} involving only $\gt_\cdot(\cdot)$:
\begin{equation}\label{fctineq(gt)2}
\gt_\al(r) \le \gt_\b(r) \exp\(-(\b-\al) \(\frac{r}{a+\int_0^{r} \gt_\b(m)\,\ud m}-1\)\).
\end{equation}
(Note that the exponent remains unchanged all the time from \eqref{fctineq(f)beflims} till now.)
\par Recall that once we have
$$\int_0^{\infty} \gt_\b(m)\,\ud m = \Est(M_1^{(\b)})<\infty,$$
then we obtain Theorem \ref{lemMen} for $\gt_\al(r)$, for $\al<\b$. This bound is going to be established by showing the rapid decay of $\gt_p$, using repeatedly \eqref{fctineq(gt)2}.  The next lemma is the first step of this procedure.
\begin{lem}[cf.~\protect{\cite[Lemma~(5.24)]{Grim}}]\label{gt<=sqrt}
For any $p<p_0$, there exists $\d(p)$ such that
$$\gt_p(r)\le \d(p)\cdot\frac{1}{\sqrt{r}}\quad\textrm{for }r>0.$$
\end{lem}
We defer proving the above lemma to the end of this section.
\par Obtaining Theorem \ref{lemMen} (being proved) from Lemma \ref{gt<=sqrt} is relatively easy. First, we deduce that for $r>0$ and $p<p_0$,
$$\int_0^{r} \gt_p(m)\,\ud m \le 2\d(p)\sqrt{r},$$
so if $r\ge a^2$, then
$$a+\int_0^{r} \gt_p(m)\,\ud m \le (2\d(p)+1)\sqrt{r}.$$
Then, using \eqref{fctineq(gt)2}, for $0\le\al<\b<p_0$, we have
\begin{align}
\int_{a^2}^{\infty} \gt_\al(r)\,\ud r &\le \int_{a^2}^{\infty} \exp\(-(\b-\al) \(\frac{r}{a+\int_0^{r} \gt_\b(m)\,\ud m}-1\)\) \,\ud r \le\\
&\le e \int_{a^2}^{\infty} \exp\biggl(-\underbrace{\frac{\b-\al}{2\d(\b)+1}}_{=C>0} \sqrt{r}\biggl) \,\ud r =\\
&= e \int_{a}^{\infty} e^{-Cx}\cdot 2x \,\ud x,
\end{align}
so
$$
\Est(M_1^{(\al)}) = \int_0^\infty \gt_\al(r)\,\ud r \le a^2 + e\int_{a}^{\infty} e^{-Cx}\cdot 2x \,\ud x < \infty,
$$
as desired. Finally, we use the finiteness of $\Est(M_1^{(\al)})$ as promised: for $r>0$ and $0\le\al<p_0$, if we take $\al<\b<p_0$, then, using \eqref{fctineq(gt)2} again,
\begin{align}
g_\al(r) \le \gt_\al(r) \le \exp\(-(\b-\al)\(\frac{r}{a+\Est(M_1^{(\b)})} -1\)\) \le e^{-\phi(\al,\b)r+\g(\al,\b)},
\end{align}
for some constants $\phi(\al,\b),\g(\al,\b)>0$.
\par Now we perform a standard estimation, aiming to rule out the additive constant $\g(\al,\b)$. For any $0<\psi_1<\phi(\al,\b)$, there exists $r_0>0$ such that for $r\ge r_0$,
$$-\phi(\al,\b)r+\g(\al,\b) \le -\psi_1 r,$$
so
$$g_\al(r)\le e^{-\psi_1 r}.$$
On the other hand, for any $r>0$, $g_p(r)$ is no greater than the probability of opening at least on edge adjacent to $o$, so $g_\al(r) \le 1-(1-\al)^{\deg(o)}<1$, where $\deg(o)$ is the degree of $o$ in the graph $G$. Hence,
$$g_\al(r)\le e^{-\psi_2(\al)r}$$
for $r\le r_0$, for some sufficiently small $\psi_2(\al)>0$. Taking $\psi = \min(\psi_1,\psi_2(\al))$ gives
$$g_\al(r) \le e^{-\psi r}$$
for any $r>0$, completing the proof of Theorem \ref{lemMen}.
\end{proof}

\par Now we are going to prove Lemma \ref{gt<=sqrt}.
\begin{proof}[Proof of Lemma \ref{gt<=sqrt}]
Assume without loss of generality that $\gt_p(r)>0$ for $r>0$.
We are going to construct sequences $(p_i)_{i=1}^\infty$ and $(r_i)_{i=1}^\infty$ such that
$$p_0>p_1>p_2>\cdots>p,\quad 0<r_1\le r_2\le\cdots$$
and such that the sequence $(\gt_{p_i}(r_i))_{i=1}^\infty$ decays rapidly. The construction is by recursion: for $i\ge 1$, having constructed $p_1,\ldots,p_i$ and $r_1,\ldots,r_i$, we put
\begin{equation}\label{recur}
r_{i+1} = r_i/g_i\quad\textrm{and}\quad p_{i+1} = p_i - 3g_i(1-\ln g_i),
\end{equation}
where $g_i = \gt_{p_i}(r_i)$. (Note that indeed, $r_{i+1}\le r_i$ and $p_{i+1}<p_i$.) The above formula may give an incorrect value of $p_{i+1}$, i.e.~not satisfying $p_{i+1}>p$ (this condition is needed because we want to bound values of $\gt_p$). In order to prevent that, we choose appropriate  values of $p_1,r_1$, using the following fact to bound the difference $p_1 - p_i$ by a small number independent of $i$. 
\begin{prop}
If we define sequence $(x_i)_{i=1}^\infty$ by $x_{i+1}=x_i^2$ for $i\ge 1$ (i.e.~$x_i=x_1^{2^{i-1}}$) with $0<x_1<1$, then
\begin{equation}\label{s(x)}
s(x_1) := \sum_{i=1}^\infty 3x_i(1-\ln x_i)
\end{equation}
is finite and $s(x_1)\tends{x_1\to 0}0$.
\end{prop}
(The idea of the proof of this fact is similar to that of estimating the sum in \eqref{smallinfsum}.)
To make use of it, we are going to bound $g_i$ by $x_i$ for any $i$ and to make $g_1$ small enough. It is done thanks to the two claims below, respectively.
\begin{clm}
If $p_1,\ldots,p_i>p$ and $r_1,\ldots,r_i>0$ are defined by \eqref{recur} with $r_1\ge a$, we have
$$g_{j+1}\le g_j^2$$
for $j=1,\ldots,i-1$.
\end{clm}
\begin{proof}
Let $j\in\{1,\ldots,i-1\}$. From \eqref{fctineq(gt)2},
\begin{align}
g_{j+1} &\le \gt_{p_j}(r_{j+1})\exp\(-(p_j-p_{j+1})\(\frac{r_{j+1}}{a+\int_0^{r_{j+1}} \gt_{p_j}(m)\,\ud m} -1\)\) \le\\
&\le g_j \exp\(1 - (p_j-p_{j+1}) \frac{r_{j+1}}{a+\int_0^{r_{j+1}} \gt_{p_j}(m)\,\ud m} \).
\end{align}
Inverse of the fraction above is estimated as follows
\begin{align}
\frac{1}{r_{j+1}} \(a+\int_0^{r_{j+1}} \gt_{p_j}(m)\,\ud m\) &\le \frac{a}{r_{j+1}} + \frac{r_j}{r_{j+1}} + \frac{1}{r_{j+1}} \int_{r_j}^{r_{j+1}} \gt_{p_j}(m)\,\ud m \le\\
&\le \frac{a}{r_{j+1}} + g_j + \frac{r_{j+1}-r_j}{r_{j+1}} \gt_{p_j}(r_j) \le
\intertext{(using $r_{j+1} = r_j/g_j$ and the monotonicity of $\gt_{p_j}(\cdot)$)}
&\le \frac{a}{r_{j+1}} + 2g_j.
\end{align}
Now, by the assumption, $r_j\ge r_1\ge a$, so $r_{j+1} = r_j/g_j \ge a/g_j$ and
$$\frac{a}{r_{j+1}} + 2g_j \le 3g_j.$$
That gives
$$
g_{j+1} \le g_j \exp\(1 - \frac{p_j-p_{j+1}}{3g_j}\) = g_j^2
$$
by the definition of $p_{j+1}$.
\end{proof}
\begin{denot1}
Put
$$M\hR=r(C\hR)$$
for $(h,R)\in\IOd$.
\end{denot1}
Note that, by Remark \ref{rembdd}, for any $p\in\NP$, $\Pr_p$-a.s.~$ M\hR<\infty$.
\begin{clm}\label{Mtight}
For any $p\in\NP$,
$$\gt_p(r)\tends{r\to\infty}0.$$
\end{clm}

\begin{proof}
First, note that it is sufficient to prove
\begin{equation}\label{convMtight}
\sup_{R\in O(d)} \Pr_p( M^{(1,R)}\ge r)\tends{r\to\infty}0,
\end{equation}
because
\begin{align}
g_p(r) &= \sup_{(h,R)\in\IOd} \Pr_p(\ho\conn S_r\textrm{ in }\G\hR\cap L) \le\\
&\le \sup_{(h,R)\in\IOd} \Pr_p(\ho\conn S_{hr}\textrm{ in }\G\hR) =\quad\textrm{(because $hr\le r$)}\\
&= \sup_{R\in O(d)} \Pr_p(o\conn S_r\textrm{ in }\G^{(1,R)}) \le\quad\textrm{(by scaling the situation)}\\
&\le \sup_{R\in O(d)} \Pr_p(M^{(1,R)}\ge r),
\end{align}
so $g_p(r)\tends{r\to\infty}0$ and, equivalently, $\gt_p(r)\tends{r\to\infty}0$ will be implied.
To prove \eqref{convMtight}, we use upper semi-continuity of the function $O(d)\ni R\mapsto \Pr_p( M^{(1,R)}\ge r)$ for any $p\in\NP$ and $r>0$. Let us fix such $p$ and $r$ and let $(R_n)_n$ be a sequence of elements of $O(d)$ convergent to some $R$. Assume without loss of generality that the cluster $C^{(1,R)}$ is bounded in the Euclidean metric and, throughout this proof, condition on it all the events by default. We are going to show that
\begin{equation}\label{incllimsupM}
\limsup_{n\to\infty}\{M^{(1,R_n)}\ge r\}\subseteq\{ M^{(1,R)}\ge r\}.
\end{equation}

\begin{df}
For any isometry $\Phi$ of $\Hd$, let $\c\Phi$ denote the unique continuous extension of $\Phi$ to $\cHd$ (which is a homeomorphism of $\cHd$---see \cite[Corollary II.8.9]{BH}).
\end{df}
Put $\Phi_n=\Phi^{(1,R)}\circ(\Phi^{(1,R_n)})^{-1}$ and assume that the event $\limsup_{n\to\infty}\{M^{(1,R_n)}\ge r\}$ occurs. Then, for infinitely many values of $n$, all the following occur:
$$M^{(1,R_n)}\ge r \then \c C^{(1,R_n)}\textrm{ intersects }\c S_r \stackrel{\c\Phi_n(\cdot)}{\then} \c C^{(1,R)}=\hc C^{(1,R)}\textrm{ intersects }\c\Phi_n(\c S_r).$$
Let for any such $n$, $x_n$ be chosen from the set $\hc C^{(1,R)}\cap\c\Phi_n(\c S_r)$. Because $\hc C^{(1,R)}$ is compact, the sequence $(x_n)_n$ (indexed by a subset of $\N_+$) has an (infinite) subsequence $(x_{n_k})_{k=1}^\infty$ convergent to some point in $\hc C^{(1,R)}$. On the other hand, note that $\c\Phi_n\tends{n\to\infty}\id_{\cHd}$ uniformly in the Euclidean metric of the disc model (see Definition \ref{dfhHd}). Hence, the distance in that metric between $x_{n_k}\in\c\Phi_{n_k}(\c S_r)$ and $\c S_r$ tends to $0$ with $k\to\infty$, so
$$\lim_{k\to\infty} x_{n_k} \in \c S_r\cap \hc C^{(1,R)} = \hc S_r\cap \hc C^{(1,R)},$$
which shows that $ M^{(1,R)}\ge r$, as desired in \ref{incllimsupM}. Now,
\begin{align}
\limsup_{n\to\infty} \Pr_p(M^{(1,R_n)}\ge r) &\le \Pr_p(\limsup_{n\to\infty}\{M^{(1,R_n)}\ge r\}) \le\quad\textrm{(by an easy exercise)}\\
&\le \Pr_p( M^{(1,R)}\ge r),
\end{align}
which means exactly the upper semi-continuity of $R\mapsto \Pr_p( M^{(1,R)}\ge r)$.
\par Next, note that because for $p\in\NP$ and $R\in O(d)$, a.s.~$M^{(1,R)}<\infty$, we have
$$\Pr_p(M^{(1,R)}\ge r) \tends{r\to\infty} 0\quad\textrm{(decreasingly)}.$$

Hence, if for $r>0$ and $\e>0$ we put
$$U_\e(r)=\{R\in O(d): \Pr_p(M^{(1,R)}\ge r)<\e\},$$
then for any fixed $\e>0$,
\begin{equation}\label{Mtight.cupUe}
\bigcup_{r\nearrow\infty} U_\e(r) = O(d).
\end{equation}
$U_\e(r)$ is always an open subset of $O(d)$ by upper semi-continuity of $R\mapsto \Pr_p( M^{(1,R)}\ge r)$, so by the compactness of $O(d)$, the union \eqref{Mtight.cupUe} is indeed finite. Moreover, because $U_\e(r)$ increases as $r$ increases, it equals $O(d)$ for some $r>0$. It means that $\sup_{R\in O(d)}\Pr_p(M^{(1,R)}\ge r)\le\e$, whence $\sup_{R\in O(d)}\Pr_p(M^{(1,R)}\ge r)\tends{r\to\infty} 0$, as desired.
\end{proof}

Now, taking any $p_1\in(p,p_0)$ and $1>x_1>0$ in \eqref{s(x)} s.t.~$s(x_1) \le p_1-p$ and taking $r_1\ge a$ so large that $\gt_{p_1}(r_1)<x_1$, we obtain for $i\ge 1$, $g_i<x_i$ (by induction). Then, in the setting of \eqref{recur},
\begin{align}
p_{i+1} = p_1 - \sum_{j=1}^i 3g_i(1-\ln g_i) > p_1 - \sum_{j=1}^i 3x_i(1-\ln x_i) \ge
\intertext{(because $x\mapsto 3x(1-\ln x)$ is increasing for $x\in(0;1]$)}
\ge p_1 - s(x_1) \ge p.
\end{align}
Once we know that the recursion \eqref{recur} is well-defined, we use the constructed sequences to prove the lemma. First, note that for $k\ge 1$,
$$r_k = r_1/(g_1g_2\cdots g_{k-1}).$$
Further, the above claim implies
\begin{align}\label{chain(g)}
g_{k-1}^2 \le g_{k-1} g_{k-2}^2 \le \cdots \le g_{k-1}g_{k-2}\cdots g_2 g_1^2 = \frac{r_1}{r_k}g_1 = \frac{\d^2}{r_k},
\end{align}
where $\d=\sqrt{r_1 g_1}$. Now, let $r\ge r_1$. We have $r_k\tends{k\to\infty}\infty$ because $\frac{r_k}{r_{k+1}}=g_k \tends{k\to\infty}0$, so for some $k$, $r_{k-1}\le r<r_k$. Then,
\begin{align}
\gt_p(r) \le \gt_{p_{k-1}}(r) \le \gt_{p_{k-1}}(r_{k-1}) = g_{k-1} \le \frac{\d}{\sqrt{r_k}} < \frac{\d}{\sqrt{r}}
\end{align}
(from \eqref{chain(g)} and the monotonicity of $\gt_p(r)$ with regard to each of $p$ and $r$), which finishes the proof.
\end{proof}

\begin{rem}\label{diffprfs}
As declared in Section \ref{prflemMen}, in this remark we summarise the differences between the proof of Theorem \ref{lemMen} and the proof of Theorem (5.4) in \cite{Grim}:
\begin{enumerate}
\item First, the skeleton structure and most of the notation of the proof here  is borrowed from \cite{Grim}. The major notation that is different here, is ``$A^\d(r)$'' and ``$S_r$'' (respectively $A_n$ and $\bdi S(s)$ in \cite{Grim}).
\item To be strict, the proper line of the proof borrowed from \cite{Grim} starts by considering the functions $f_p$ instead of $g_p$ or $\gt_p$, although the functional inequality \eqref{fctineq(f)} involves both functions $f_\cdot(\cdot)$ and $\gt_\cdot(\cdot)$. In fact, each of the functions $f_\cdot(\cdot)$, $\gt_\cdot(\cdot)$ and $g_\cdot(\cdot)$ is a counterpart of the function $g_\cdot(\cdot)$ from \cite{Grim} at some stage of the proof. After proving inequality \eqref{fctineq(f)}, we pass to a couple of limits with it in order to obtain inequality \eqref{fctineq(gt)1} involving only $\gt_\cdot(\cdot)$ (the step not present in \cite{Grim}). This form is  needed to perform the repeated use of inequality \eqref{fctineq(gt)1} at the end of the proof of Theorem \ref{lemMen}.
\item Obviously, the geometry used here is much different from that in \cite{Grim}. In fact, we analyse the percolation cluster in $\G\hR\cap L_\d$ using the pseudometric $d_\hbd$ (in place of the graph metric $\d$ in \cite{Grim}). Consequently, the set $\RhR$ of possible values of the random variables $\rho$ in Lemma \ref{saus} is much richer than $\N$, the respective set for the graph $\Z^d$. Moreover, the functions $\gt_p$ arise from the percolation process on the whole $\G\hR\cap L$, so the distribution of the random variables $M_i$ is not necessarily discrete. That cause the need for using integrals instead of sums, when concerned with those random variables, especially in the proof of Corollary \ref{corsaus}. All that leads also to a few other minor technical differences between the proof here and the proof in \cite{Grim}.
\item The author tried to clarify the use of the assumption on $\sum_{i=1}^k r_i$ in Lemma \ref{saus} and why Wald's equation can be used in the proof of Lemma \ref{lemEN>=frac}, which could be found quite hidden in \cite{Grim}.
\item The proof of Lemma \ref{gt<=sqrt} itself has a little changed structure (compared to the proof of Lemma (5.24) in \cite{Grim}) and contains a proof of the convergence $\gt_p(r)\tends{r\to\infty}0$ (Claim \ref{Mtight}).
\end{enumerate}
\end{rem}

\begin{rem}\label{whyLd}
In order to prove Theorem \ref{lemMen}, one could try to consider the percolation processes on the whole graph $\G\hR\cap L$ (without restricting it to $L_\d$) in order to obtain functional inequality similar to \eqref{fctineq(gt)1}, involving only one function. That approach caused many difficulties to the author, some of which have not been overcome. Restricting the situation to $L_\d$ makes the event $A^\d$ depend on the states of only finitely many edges. That allows e.g.\ to condition the event $A^\d(r)\cap B$ on the family of events $\{\Gamma\textrm{ a witness for }B\}$, where $\Gamma$ runs over a countable set (in the proof of Lemma \ref{saus}) or to use BK inequality and Russo's formula.
\end{rem}

\section*{Acknowledgements}

\end{document}